\documentclass[10pt]{article}

\usepackage[svgnames]{xcolor}

\usepackage{latexsym,amssymb,geometry,datetime,aliascnt,bm,graphicx,algorithmic,algorithm,amsthm,amssymb,amsmath,amsfonts,rotate,geometry,amsthm,amssymb,amsmath,amsfonts,rotate,lmodern,amsmath,ifthen,eurosym,graphicx,algorithmic,algorithm,aliascnt,stmaryrd,subcaption,wrapfig,cite,latexsym,amssymb,url,amsmath,latexsym,amsthm,amsfonts,amssymb,ifthen,color,wrapfig,rotate,lmodern,aliascnt,datetime,graphicx,algorithmic,lmodern,algorithm,enumerate,enumitem,todonotes,bm,subcaption,tabularx,colortbl,xspace,amsmath,ifthen,eurosym,wrapfig,cite,subcaption,url,fullpage}

\usepackage{tikz}
\usetikzlibrary{arrows.meta,patterns}
\usetikzlibrary{calc,3d}
\usetikzlibrary{intersections,decorations.pathmorphing,shapes,decorations.pathreplacing,fit}

\usepackage[pdftex,
backref=true,
plainpages = true,
pdfpagelabels,
hyperfootnotes=true,
pdfpagemode=FullScreen,
bookmarks=true,
bookmarksopen = true,
bookmarksnumbered = true,
breaklinks = true,
hyperfigures,
linktocpage,
pagebackref,
urlcolor = magenta,
urlcolor = MidnightBlack,
anchorcolor = green,
hyperindex = true,
colorlinks = true,
linkcolor = black!30!blue,
citecolor = black!30!green]{hyperref}

\RequirePackage[T1]{fontenc}
\RequirePackage[utf8]{inputenc}
\usepackage[english]{babel}
\RequirePackage{stmaryrd}
\usepackage{textcomp}
\DeclareUnicodeCharacter{2192}{\ifmmode\to\else\textrightarrow\fi}
\DeclareUnicodeCharacter{2203}{\ensuremath\exists}
\DeclareUnicodeCharacter{183}{\cdot}
\DeclareUnicodeCharacter{2200}{\forall}
\DeclareUnicodeCharacter{2264}{\leq}
\DeclareUnicodeCharacter{2265}{\geq}
\DeclareUnicodeCharacter{8614}{\mathbin{\mapsto}}
\DeclareUnicodeCharacter{8656}{\Leftarrow}
\DeclareUnicodeCharacter{8657}{\Uparrow}
\DeclareUnicodeCharacter{8658}{\Rightarrow}
\DeclareUnicodeCharacter{8659}{\Downarrow}
\DeclareUnicodeCharacter{8669}{\rightsquigarrow}
\newcommand{\eqdef}{\stackrel{{\scriptsize\rm def}}{=}}
\DeclareUnicodeCharacter{8797}{\eqdef}
\DeclareUnicodeCharacter{8870}{\vdash}
\DeclareUnicodeCharacter{8873}{\Vdash}
\DeclareUnicodeCharacter{22A7}{\models}
\DeclareUnicodeCharacter{9121}{\lceil}
\DeclareUnicodeCharacter{9123}{\lfloor}
\DeclareUnicodeCharacter{9124}{\rceil}
\DeclareUnicodeCharacter{222A}{\cup}
\DeclareUnicodeCharacter{2208}{\in}
\DeclareUnicodeCharacter{9126}{\rfloor}
\DeclareUnicodeCharacter{2286}{\subseteq}
\DeclareUnicodeCharacter{9655}{\triangleright}
\DeclareUnicodeCharacter{9665}{\triangleleft}
\DeclareUnicodeCharacter{9671}{\diamond}
\DeclareUnicodeCharacter{9675}{\circ}
\DeclareUnicodeCharacter{10178}{\bot}
\DeclareUnicodeCharacter{10214}{} 
\DeclareUnicodeCharacter{10215}{} 
\DeclareUnicodeCharacter{10229}{\longleftarrow}
\DeclareUnicodeCharacter{10230}{\longrightarrow}
\DeclareUnicodeCharacter{10231}{\longleftrightarrow}
\DeclareUnicodeCharacter{10232}{\Longleftarrow}
\DeclareUnicodeCharacter{10233}{\Longrightarrow}
\DeclareUnicodeCharacter{10234}{\Longleftrightarrow}
\DeclareUnicodeCharacter{10236}{\longmapsto}
\DeclareUnicodeCharacter{10238}{\Longmapsto} 
\DeclareUnicodeCharacter{10503}{\Mapsto}    
\DeclareUnicodeCharacter{10971}{\mathrel{\not\hspace{-0.2em}\cap}}
\DeclareUnicodeCharacter{65294}{\ldotp}
\DeclareUnicodeCharacter{65372}{\mid}

\tikzstyle{ipe stylesheet} = [
  ipe import,
  even odd rule,
  line join=round,
  line cap=butt,
  ipe pen normal/.style={line width=0.4},
  ipe pen heavier/.style={line width=0.8},
  ipe pen fat/.style={line width=1.2},
  ipe pen ultrafat/.style={line width=2},
  ipe pen normal,
  ipe mark normal/.style={ipe mark scale=3},
  ipe mark large/.style={ipe mark scale=5},
  ipe mark small/.style={ipe mark scale=2},
  ipe mark tiny/.style={ipe mark scale=1.1},
  ipe mark normal,
  /pgf/arrow keys/.cd,
  ipe arrow normal/.style={scale=7},
  ipe arrow large/.style={scale=10},
  ipe arrow small/.style={scale=5},
  ipe arrow tiny/.style={scale=3},
  ipe arrow normal,
  /tikz/.cd,
  ipe arrows, 
  <->/.tip = ipe normal,
  ipe dash normal/.style={dash pattern=},
  ipe dash dotted/.style={dash pattern=on 1bp off 3bp},
  ipe dash dashed/.style={dash pattern=on 4bp off 4bp},
  ipe dash dash dotted/.style={dash pattern=on 4bp off 2bp on 1bp off 2bp},
  ipe dash dash dot dotted/.style={dash pattern=on 4bp off 2bp on 1bp off 2bp on 1bp off 2bp},
  ipe dash normal,
  ipe node/.append style={font=\normalsize},
  ipe stretch normal/.style={ipe node stretch=1},
  ipe stretch normal,
  ipe opacity 10/.style={opacity=0.1},
  ipe opacity 30/.style={opacity=0.3},
  ipe opacity 50/.style={opacity=0.5},
  ipe opacity 75/.style={opacity=0.75},
  ipe opacity opaque/.style={opacity=1},
  ipe opacity opaque,
]

\everymath{\color{MidnightBlue}}

\usepackage{color}
\definecolor{MidnightBlack}{rgb}{0.1,0.1,0.36}
\definecolor{MidnightBlue}{rgb}{0.1,0.1,0.41}
\definecolor{Black}{rgb}{0,0, 0}
\definecolor{Blue}{rgb}{0, 0 ,1}
\definecolor{Red}{rgb}{1, 0 ,0}
\definecolor{White}{rgb}{1, 1, 1}
\definecolor{Grey}{rgb}{.6, .6, .6}
\definecolor{Mygreen}{rgb}{.0, .6, .0}
\definecolor{Yellow}{rgb}{.55,.55,0}
\definecolor{Mustard}{rgb}{1.0, 0.86, 0.35}
\definecolor{applegreen}{rgb}{0.55, 0.71, 0.0}
\definecolor{darkturquoise}{rgb}{0.0, 0.81, 0.82}
\definecolor{celestialblue}{rgb}{0.29, 0.59, 0.82}
\definecolor{green_yellow}{rgb}{0.68, 1.0, 0.18}
\definecolor{crimsonglory}{rgb}{0.75, 0.0, 0.2}
\definecolor{darkmagenta}{rgb}{0.30, 0.0, 0.30}
\definecolor{internationalorange}{rgb}{1.0, 0.31, 0.0}
\definecolor{darkorange}{rgb}{1.0, 0.55, 0.0}

\newcommand{\red}[1]{{\color{Red}#1}}

\newcommand{\green}[1]{{\color{Mygreen}#1}}

\newcommand{\tw}{{\sf tw}}

\newcommand{\remove}[1]{}

\newcounter{func}

\newcommand{\funref}[1]{\hyperref[#1]{f_{\ref*{#1}}}} 
\newcounter{con}

\newcommand{\conref}[1]{\hyperref[#1]{c_{\ref*{#1}}}} 

\newcommand{\mynewtheorem}[2]{
	\newaliascnt{#1}{dummy}
	\newtheorem{#1}[#1]{#2}
	\aliascntresetthe{#1}
}

\theoremstyle{plain}
\mynewtheorem{theorem}{Theorem}
\mynewtheorem{proposition}{Proposition}
\mynewtheorem{corollary}{Corollary}
\mynewtheorem{lemma}{Lemma}
\theoremstyle{definition}
\mynewtheorem{definition}{Definition}
\theoremstyle{remark}
\mynewtheorem{sublemma}{Sublemma}
\mynewtheorem{claim}{Claim}
\mynewtheorem{remark}{Remark}
\mynewtheorem{fact}{Fact}
\mynewtheorem{conjecture}{Conjecture}
\mynewtheorem{question}{Question}
\mynewtheorem{observation}{Observation}
\mynewtheorem{problem}{Problem}
\mynewtheorem{note}{Note}

\usepackage{xspace}

\newcommand{\sqgm}{\mbox{\rm \ttfamily\fontseries{l}\selectfont SQG\fontseries{b}\selectfont M}}
\newcommand{\sqgc}{\mbox{\rm \ttfamily\fontseries{l}\selectfont SQG\fontseries{b}\selectfont C}}
\newcommand{\bcg}{{\bf bcg}}
\newcommand{\sm}{\setminus}     
\newcommand{\intv}[2]{\left \llbracket #1, #2 \right \rrbracket}
\newcommand{\es}{\emptyset}

\newcommand*{\myproofname}{My proof}
\newenvironment{myproof}[1][\myproofname]{\begin{proof}[#1]}{\end{proof}}

\newcommand{\Acal}{\mathcal{A}}
\newcommand{\Ocal}{\mathcal{O}}
\newcommand{\Gcal}{\mathcal{G}}
\newcommand{\Scal}{\mathcal{S}}
\newcommand{\Xcal}{\mathcal{X}}

\begin{document}



\title{Contraction Bidimensionality of Geometric \\ Intersection Graphs\thanks{Corresponding author: Dimitrios M. Thilikos, Email: \texttt{sedthilk@thilikos.info}\,.}~$^{,}$\thanks{An extended abstract of this article appeared in the \emph{Proceedings of the 12th International Symposium on Parameterized and Exact Computation, IPEC 2017, September 6-8, 2017, Vienna, Austria} \cite{BasteT17contr}. The first author was supported by ANR projects DEMOGRAPH (ANR-16-CE40-0028). The second author was supported  by   the ANR projects DEMOGRAPH (ANR-16-CE40-0028), ESIGMA (ANR-17-CE23-0010), and the French-German Collaboration ANR/DFG Project UTMA (ANR-20-CE92-0027).}}

\author{Julien Baste\thanks{Univ Lille, Centrale
	Lille, CRIStAL, Lille, France,  Email: \texttt{julien.baste@univ-lille.fr}\,.}
\and  Dimitrios M. Thilikos$^{*,}$\thanks{LIRMM, Univ Montpellier, CNRS, Montpellier, France.}}
\date{\empty}

\maketitle
%
%
%
%

\begin{abstract}
\noindent Given a graph $G$, we define ${\bf bcg}(G)$ as the minimum $k$ for which $G$ can be contracted to the uniformly triangulated grid $\Gamma_{k}$. A graph class ${\cal G}$ has the SQG{\bf C} property 
if  every graph $G\in{\cal G}$ has treewidth $\Ocal({\bf bcg}(G)^{c})$ for some $1\leq c<2$.
The  SQG{\bf C} property is important for algorithm design as it defines the 
applicability horizon of a series of meta-algorithmic results, in the framework of bidimensionality theory, 
related to fast parameterized algorithms, kernelization, and approximation schemes. These results apply to 
a wide family of problems,  namely problems that are  {\sl contraction-bidimensional}.
{Our main combinatorial result reveals a wide  family of graph classes that satisfy the 
SQG{\bf C} property. This family includes, in particular, bounded-degree string graphs.}
This considerably extends the applicability of bidimensionality theory 
for  contraction bidimensional problems.
 \end{abstract}

\medskip

\section{Introduction}
\label{intrintr}
 
{\sl Treewidth} is one of most well-studied parameters in graph algorithms. It serves as a measure of how close a graph is to the topological structure  of a tree (see \autoref{defprel} for the formal definition). 
Gavril is the first to introduce the concept in~\cite{Ga1974}
but 
it obtained 
 its name in the second paper of the Graph Minors series of Robertson and Seymour 
in~\cite{RoSe1986-II}.
Treewidth has extensively used in graph algorithm design due to the fact that 
a wide class of intractable problems in graphs becomes tractable when restricted on graphs of bounded treewidth~\cite{ArLaSe1991,Co1990,Co1997}.
Before we present some key combinatorial properties of treewidth, we need some definitions.

\paragraph{Graph contractions and minors.}

Our first aim is to define two parameterized versions of the contraction relation on graphs.

\begin{definition}[Contractions]\label{allmincon}
Given a non-negative integer $c$, 
two graphs $H$ and $G$, and 
a surjection $\sigma:V(G)\rightarrow V(H)$ we write $H\leq _{\sigma}^{c} G$
if 
\begin{itemize}
\item for every $x\in V(H)$, the graph $G[\sigma^{-1}(x)]$ is a non-empty graph  (i.e., a graph with at least one vertex) {of {\sl diameter} at most $c$} and 
%
\item  for every $x,y\in V(H),$ $\{x,y\}\in E(H) \iff G[\sigma^{-1}(x)\cup \sigma^{-1}(y)]$ is connected.
\end{itemize}
We say that $H$ is a {\em $c$-diameter contraction} of $G$ if there exists a surjection $\sigma:V(G)\rightarrow V(H)$ such that $H\leq_{\sigma}^{c}G$ and we write  this $H\leq^{c}  G$.
Moreover, if $\sigma$ is such that  for every $x \in V(H)$, 
$|\sigma^{-1}(x)|\leq c'$, 
then we say that $H$ is a {\em $c'$-size contraction} of $G$, and we write $H\leq^{(c')}  G$.  Given two graphs $G$ and $H$, if there exists an integer $c$ such that  $H\leq^{c}G$, then we say that $H$ is a {\em contraction} of $G$, and we write  $H\leq G$.
Moreover, if there exists a subgraph $G'$ of $G$ such that $H \leq G'$, we say that $H$ is a \emph{minor} of $G$ and we write this $H\preceq G$. Given a graph $H$, we denote by ${\sf excl}(H)$ the class of graphs that exclude $H$ as a minor.
\end{definition}

\subsection{Combinatorics of treewidth}

One of the most celebrated structural results on treewidth  is the following:
\begin{proposition}
\label{mainexcl}
There is a function $f: \Bbb{N}\rightarrow\Bbb{N}$ such that 
every  graph excluding a $(k\times k)$-grid as a minor has 
treewidth at most $f(k)$.
\end{proposition}
A proof of \autoref{mainexcl} appeared for the first time by Robertson and Seymour 
in~\cite{RoSe1986-V}. Other proofs, with better bounds to the function $f$, appeared 
in~\cite{RoSeTh1994} and later in~\cite{DiJeGoTh1999} (see also~\cite{KaKo2012,LeSe2015}). Currently, 
the best bound for $f$ is due to  Chuzhoy, who recently proved in~\cite{ChuzhoyT21towar}
that $f(k)=k^{{9}}\cdot \log^{\Ocal(1)}k$. On the other hand, it is possible to 
show that \autoref{mainexcl} is not correct when $f(k)=\Ocal(k^{2}\cdot\log k)$ 
(see~\cite{Th2012}).

The  potential of \autoref{mainexcl} on graph algorithms has been 
capitalized by the {\sl theory of bidimensionality} that
was introduced  in~\cite{DeFoHaTh2005} and has been   further  developed 
in~\cite{DeHa2005,DeHa2008theb,DeFoHaTh2005a,FoLoRaSa2011,FoLoSaTh2010,GrKoTh2014,FoGoTh2011,DeHaTh06,FoGoTh2009,FoLoSa2012}. This theory offered general techniques for designing efficient fixed-parameter algorithms and approximation schemes for {\sf NP}-hard graph problems in broad classes of graphs (see~\cite{DeFoHaTh2013,DeHa2008,FoDeHa2015,De2010,DeHa2004}). 
In order to present the result of this paper we  first give  a brief presentation of this theory and of its applicability.

\paragraph{Optimization parameters and bidimensionality.}

A {\em graph parameter}  is a function ${\bf p}$ mapping graphs to non-negative integers.
We say that {\bf p} is a {\em minimization graph parameter} if
${\bf p}(G)=\min\{k\mid \exists S\subseteq V(G): |S|\leq k \mbox{ and } \phi(G,S)={\tt true}\},$   where $\phi$ is a some predicate on $G$ and $S$.    Similarly, we say that {\bf p} is a {\em maximization graph parameter} if in the above definition we we replace $\min$ and $\leq$ by $\max$ and $\geq$ respectivelly.
Minimization or maximization parameters are briefly called {\em optimization parameters}.

\begin{definition}[Bidimensionality]
Given two real functions $f$ and $g$, we use the term $f \gtrsim g$ to denote that 
$f(x)\geq g(x)- o(g(x))$.
A graph parameter {\bf p} is {\em minor-closed} (resp. {\em contraction-closed}) when $H\preceq G\Rightarrow{\bf p}(H)\leq {\bf p}(G)$ (resp.  $H\leq  G\Rightarrow{\bf p}(H)\leq {\bf p}(G)$). 
We can now give the two following definitions:
%

\begin{center}
\begin{minipage}{6cm}
{\bf p} is {\em minor-bidimensional} if  \smallskip
\begin{itemize}

\item ${\bf p}$ is minor-closed, 
and 
\item  ${\bf p}(\boxplus_k)\gtrsim \delta\cdot k^2.$
\end{itemize}
\end{minipage}
~~~~~~~~~~~~~~~~
\begin{minipage}{6cm}
{\bf p} is {\em contraction-bidimensional}  if   \smallskip
\begin{itemize}

\item ${\bf p}$ is contraction-closed, 
and

\item  ${\bf p}(\Gamma_{k})\gtrsim \delta\cdot k^2.$
\end{itemize}
\end{minipage}
\end{center}

\noindent for some $\delta>0$. In the above definitions, we use $\boxplus_k$ for  the $(k\times k)$-grid 
and  $\Gamma_{k}$ for  the uniformly triangulated $(k\times k)$-grid (see Figure~\ref{ddsdfgdfgdfg_sdfgdfgdfgdfg}). 
If ${\bf p}$ is a minimization (resp. maximization) graph parameter, we denote by $\Pi_{\bf p}$ the problem
that, given a graph $G$ and a non-negative integer $k$, asks whether ${\bf p}(G)\leq k$ (resp. ${\bf p}(G)\geq k$). We say that a problem is {\em minor/contraction-bidimensional} if it is $\Pi_{\bf p}$
for some bidimensional optimization parameter ${\bf p}$.
\end{definition}

\begin{figure}[h]
\begin{center}
\scalebox{.42}{\includegraphics{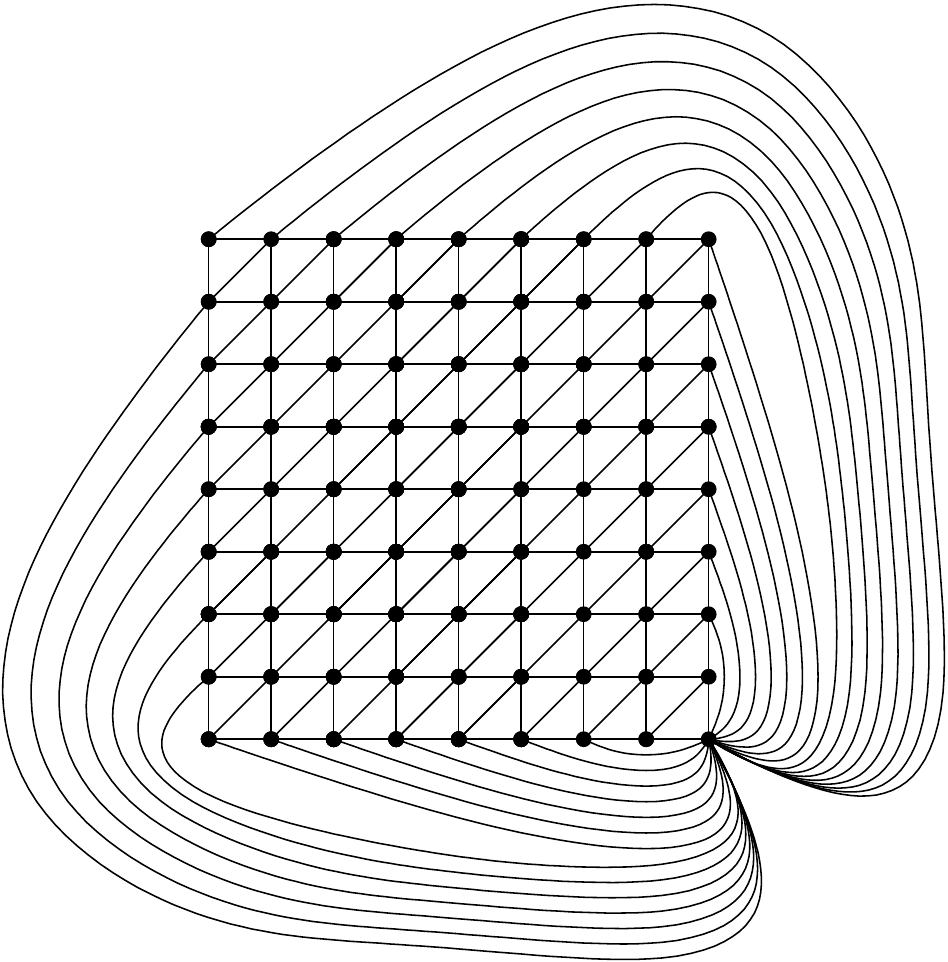}}
\end{center}
\caption{The graph $\Gamma_{9}$.}
\label{ddsdfgdfgdfg_sdfgdfgdfgdfg}
\end{figure}

A (non exhaustive) list of minor-bidimensional problems is:
  {{\sc  Vertex Cover}, {\sc  Feedback Vertex Set},  {\sc Longest Cycle},   {\sc  Longest Path}, {\sc  Cycle Packing},  {\sc Path Packing},   {\sc Diamond Hitting Set},  {\sc Minimum Maximal Matching}, {\sc Face Cover},  and {\sc Max Bounded Degree Connected Subgraph}}.
%
%
Some problems that are contraction-bidimensional (but not minor-bidimensional)
are
{\sl   {\sc  Connected Vertex Cover}, {\sc  Dominating Set}, {\sc  Connected Dominating Set},  {\sc  Connected Feedback Vertex Set}, {\sc  Induced Matching},  {\sc   Induced Cycle Packing},  {\sc  Cycle Domination},  {\sc  Connected Cycle Domination},  {\sc  $d$-Scattered Set}, {\sc  Induced Path Packing}, {\sc $r$-Center}, {\sc connected $r$-Center},  {\sc  Connected Diamond Hitting Set}, {\sc  Unweighted TSP Tour}} (see~\cite{DeHa2008theb,FoDeHa2015,Thilikos15bidi}).

\paragraph{Subquadratic grid minor/contraction property.} 

In order to present the meta-algorithmic potential of bidimensionality theory we need to define some {\sl property on graph classes} that defines the horizon of its applicability. 

\begin{definition}[SQG{\bf C} and SQG{\bf C}]
Let ${\cal G}$ be a graph class. We  say that ${\cal G}$  
 has the {\em subquadratic grid minor property} (SQG{\bf M} property for short) if there exist a constant  $1 \leq c < 2$ such that every graph $G \in {\cal G}$ which excludes $\boxplus_{t}$ as a minor, for some integer $t$, has treewidth $\Ocal(t^c)$. In other words, this property holds for ${\cal G}$ if \autoref{mainexcl}
can be proven for a sub-quadratic $f$ on the graphs of ${\cal G}$.

 Similarly, we  say that ${\cal G}$  
 has the {\em subquadratic grid contraction property} (SQG{\bf C} property for short) if there exist a constant  $1 \leq c < 2$ such that every graph $G \in {\cal G}$ which excludes $\Gamma_{t}$ as a contraction, for some integer $t$, has treewidth $\Ocal(t^c)$. 
For brevity we say that ${\cal G}\in\sqgm(c)$ (resp. ${\cal G}\in\sqgc(c)$) if ${\cal G}$  
 has the SQG{\bf M} (resp SQG{\bf C}) property for $c$. 
Notice that $\sqgc(c)\subseteq \sqgm(c)$ for every $1 \leq c < 2$.
\end{definition} 

\subsection{Algorithmic implications}

\label{resbidi}
The meta-algorithmic consequences of bidimensionality theory are summarised as follows.
 Let ${\cal G}\in\sqgm(c)$, for $1\leq c<2$, and let {\bf p} be a minor-bidimensional-optimization parameter.\medskip

\begin{itemize}
\item[{\bf [A]}]  As it was observed in~\cite{DeFoHaTh2005},  the problem $\Pi_{\bf p}$ can be solved in $2^{o({k})}\cdot n^{\Ocal(1)}$ steps on ${\cal G}$, 
given that the computation of ${\bf p}$ can be done in 
$2^{\Ocal({\bf tw}(G))}\cdot n^{\Ocal(1)}$ steps (here ${\bf tw}(G)$ is the treewidth of the input graph $G$). This last condition can 
be implied by a purely meta-algorithmic condition that is  based on some 
variant of {\sl Modal Logic}~\cite{Pi2011}. 
There is a wealth of results that yield the last  condition for various 
optimization problems either in  classes satisfying the SQG{\bf M} propety~\cite{RuSaTh2014,DoFoTh2008,DoFoTh2006,DoFoTh2006,RuSaTh2012} or to general graphs~\cite{CyNePiPiRoWo2011,BoCyKrNe2015,FoLoSa2014}. \medskip

\item[{\bf [B]}]    As it was shown in~\cite{FoLoSaTh2010} (see also~\cite{FoLoSa2016}), when the predicate $\phi$ 
can be expressed in Counting Monadic Second Order Logic (CMSOL) and ${\bf p}$ satisfies some additional combinatorial property called {\em separability}, then the problem  $\Pi_{\bf p}$ 
admits a {\sl linear kernel}, that is a polynomial-time algorithm that transforms $(G,k)$
to an equivalent instance $(G',k')$ of $\Pi_{\bf p}$ where $G'$ has size $\Ocal(k)$ and $k'\leq k$.\medskip

\item[{\bf [C]}]    It was proved in~\cite{FoLoRaSa2011} (see also~\cite{FominLS18exclu} and~\cite{FominL0Z20appro}),  that the problem of computing ${\bf p}(G)$ for $G\in{\cal G}$
admits a {\sl Efficient Polynomial Approximation Scheme}  (EPTAS) --- that is an $\epsilon$-approximation
algorithm running in $f(\frac{1}{\epsilon})\cdot n^{\Ocal(1)}$ steps --- given that ${\cal G}$ is hereditary and ${\bf p}$ satisfies the  separability property and some reducibility property (related to CMSOL expresibility).
\end{itemize}

\begin{figure}[h]
\begin{center}
\scalebox{.22}{\includegraphics{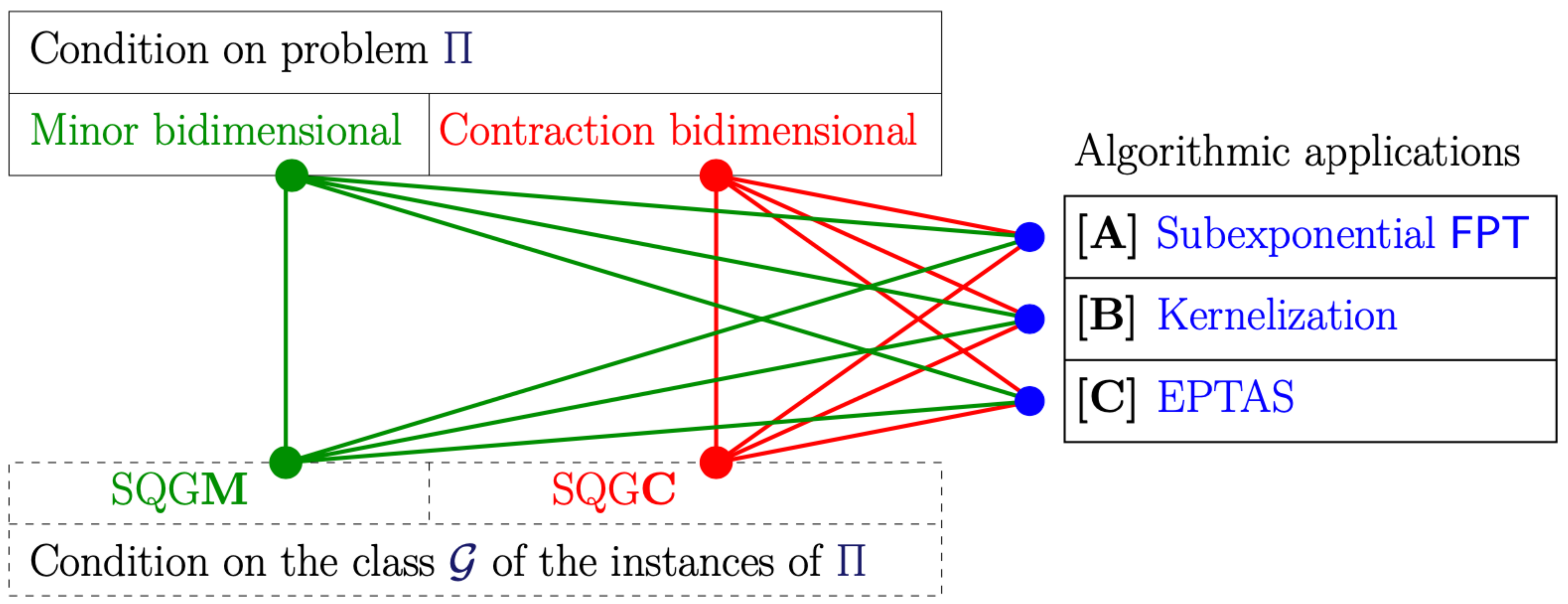}}
\end{center}
\caption{The applicability of bidimensionality theory. The \green{green} lines correspond 
the consequences \cite{GrKoTh2014}} while the \red{red} lines correspond to the result of this paper.
\label{bidipic}
\end{figure}

All above results have their counterparts for {\sl contraction-bidimensional}  problems with the difference that  one should  instead demand that ${\cal G}\in \sqgc(c)$.
Clearly, the applicability of all above results is delimited by the  SQG{\bf M}/SQG{\bf C} property. 
This is schematically depicted in Figure~\ref{bidipic}, where the \green{green} triangles triangles indicate the applicability of minor-bidimensionality and 
the \red{red} triangle   indicate  the applicability of contraction-bidimensionality.
The aforementioned $\Omega(k^2\cdot \log k)$ lower bound to the function $f$ of \autoref{mainexcl}, indicates that $\sqgm(c)$ does not contain all graphs  (given that $c<2$).

As an example we mention the well known $d$-{\em Domination Set} problem (for some $d\geq 1$), asking whether 
a graph $G$ has a set $S$ of at most $k$ vertices such that every vertex in $G$ is within distance at most $d$ from some vertex of $S$.  $d$-{\em Domination Set}  is contraction bidimensional problem 
that satisfies the additional meta-algorithmic conditions in {\bf [A]}, {\bf [B]}, and {\bf [C]}. This implies 
that it can be solved in $2^{O(\sqrt{k})}\cdot n$ time, it admits a linear kernel, and its optimization version admits an EPTAS on every graph class that has the SQG{\bf C}  property.

The emerging  direction of research is to detect the most general classes in $\sqgm(c)$
and $\sqgc(c)$.
Concerning  the SQG{\bf M} property, the following result was proven in~\cite{DeHa2008}.

\begin{proposition}
\label{j9ione}
For every graph $H$, ${\sf excl}(H)\in\sqgm(1)$.
\end{proposition}

A graph $H$ is an {\em apex graph} if it contains a vertex whose removal from $H$ results to a planar graph. For the SQG{\bf C} property, the following counterpart of \autoref{j9ione}   was  proven in~\cite{FoGoTh2011}.

\begin{proposition}
\label{jd9iodne}
For every apex graph $H$, ${\sf excl}(H)\in\sqgc(1)$.
\end{proposition}

Notice that both above results concern graph classes that are defined by excluding some graph as a minor. For such graphs, \autoref{jd9iodne} is indeed optimal. To see this, consider $K_{h}$-minor free graphs where $h\geq 6$ (these graphs are not apex graphs). Such classes do not satisfy the SQG{\bf C} property:
take $\Gamma_{k}$, add a new vertex, and make it adjacent,  with all its vertices. The resulting 
graph excludes  $\Gamma_{k}$ as a contraction and has treewidth $>k$.

\subsection{String graphs}

An important step extending the applicability of bidimensionality theory further than
$H$-minor free graphs, was done in~\cite{FoLoSa2012} (see also~\cite{FominLS18exclu}). 

\begin{definition}[String graphs, map graphs, and unit disk graphs]\label{fnolki}
{\em Unit disk} graphs are intersections 
graphs of unit disks in the plane and  {\em map} graphs are intersection 
graphs of face boundaries of planar graph embeddings.
We denote by ${\cal U}_{d}$ the set of unit disk graphs (resp. of ${\cal M}_{d}$ map graphs) of maximum degree $d$.
\end{definition}
 The following was proved in~\cite{FoLoSa2012,FominLS18exclu}.

\begin{proposition}
\label{fi9ougs}
For every positive integer $d$, ${\cal U}_{d}\in \sqgm(1)$ and  ${\cal M}_{d}\in \sqgm(1)$.
\end{proposition}
\noindent \autoref{fi9ougs} was further extended for 
intersection graphs of more general geometric objects (in 2 dimensions) in~\cite{GrKoTh2014}. To explain the results of~\cite{GrKoTh2014}
we need to define a more general model of intersection graphs.

\begin{definition}(String graphs)\label{stringgg}
Let ${\cal L}=\{L_{1},\ldots,L_{k}\}$ be a collection of lines in the plane.
We say that ${\cal L}$ is {\em normal} if
there is no point belonging to more than two lines.
The {\em intersection graph $G_{\cal L}$} of ${\cal L}$, is the graph whose vertex set is ${\cal L}$
and where, for each $i,j$ where $1\leq i< j\leq k$, the edge $\{L_{i},L_{j}\}$ has multiplicity $|L_{1}\cap L_{2}|$. We denote by ${\cal S}_{d}$ the set containing every graph ${G}_{\cal L}$ where ${\cal L}$ is a normal collection of lines in the plane
and where each vertex of $G_{\cal L}$ has edge-degree at most $d$. i.e.,  
is incident to at most $d$ edges.
We call ${\cal S}_{d}$ {\em string graphs with edge-degree bounded by $d$}.
\end{definition}

 It is easy to observe that 
 ${\cal U}_{d}\cup {\cal M}_{d}\subseteq {\cal S}_{f(d)}$ for some quadratic function $f$.
 Indeed, given a graph $G$ in ${\cal U}_{d}$, for each unit disk of its representation in the plane, we can create a string that corresponds to the perimeter of the disk.
 As all the disks are of the same size, the intersection graph of the strings is homeomorphic to $G$.
 The same applies for map graphs by considering the boundaries of the faces and creating a string for each of them.
 Moreover, apart from the classes considered in~\cite{FoLoSa2012}, ${\cal S}_{d}$ includes a much wider variety of classes of intersection graphs~\cite{GrKoTh2014}. As an example, consider ${\cal C}_{d,\alpha}$ as the class of all graphs that are intersection graphs of $\alpha$-convex  bodies%
 \footnote{We call a set of points in the plane a {\em body} if it is homeomorphic to the closed disk $\{(x, y) \mid  x^2 + y^2 \leq 1\}$. A 2-dimensional  body $B$ is a {\em $\alpha$-convex}  if every two of its points can be the extremes of  a line $L$  consisting of  $\alpha$ straight lines and where  $L\subseteq B$. {\em Convex} bodies are exactly the $1$-convex bodies.} 
in the plane and have {edge}-degree at most $d$.
 In~\cite{GrKoTh2014}, it was proven that  ${\cal C}_{d,\alpha}\subseteq {\cal S}_{c}$ where $c$ depends (polynomially) on $d$ and $\alpha$. Another interesting class from~\cite{GrKoTh2014} is ${\cal F}_{H,\alpha}$ containing all  $H$-subgraph free intersection graphs of $\alpha$-fat\footnote{A collection of convex bodies in the pane is 
{\em $\alpha $-fat} if the ratio between the maximum and the minimum radius of a circle 
where all bodies of the collection can be circumscribed and inscribed
respectively,  is upper bounded by $a$.} families of convex bodies. Notice ${\cal U}_{d}$ can be seen as a special case of both  ${\cal C}_{d,\alpha}$ and ${\cal F}_{H,\alpha}$. (See~\cite{Ma2014} for other examples of classes included in ${\cal S}_{d}$.)

\begin{figure}[ht]
  \centering
\scalebox{.16}{\includegraphics{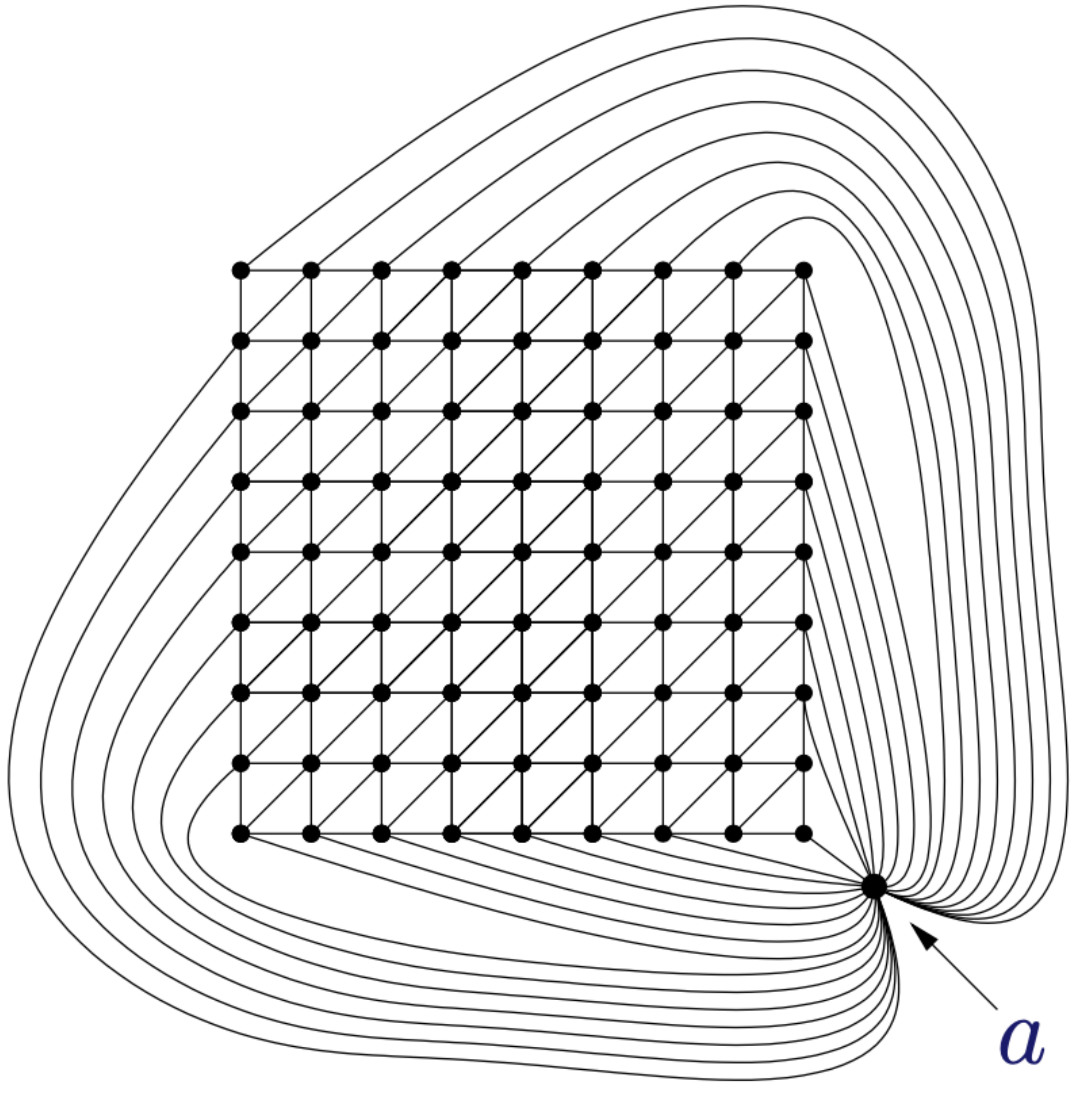}}
\caption{A graphical representation of the definition of $\mathcal{G}^{(c_1,c_2)}$.}
\label{wahyufoi}
\end{figure}

\subsection{Our contribution}

\paragraph{Graph class extensions.}

\begin{definition}[$(c_1,c_2)$-extension]\label{planall}
Given a class of graph $\Gcal$ and two integers  $c_1$ and $c_2$, we define the {\em $(c_1,c_2)$-extension} of ${\cal G}$, denoted by $\Gcal^{(c_1,c_2)}$, as the set containing every graph $H$ such that there exist a graph $G \in \Gcal$ and a graph $J$ that satisfy $G \leq^{(c_1)} J \mbox{ and } H \leq^{c_2} J$ (see \autoref{wahyufoi} for a visualization of this construction). Keep in mind that  $\Gcal^{(c_1,c_2)}$ and  $\Gcal^{(c_2,c_1)}$ are two different graph classes. We also 
denote by ${\cal P}$ the class of all planar graphs. 
\end{definition}

Using the above notation, the two combinatorial results in~\cite{GrKoTh2014}
can be rewritten as follows:

\begin{proposition}
\label{j9iossnde}
Let $c_{1}\geq 1$ and $c_{2}\geq 0$ be two integers. If ${\cal G}\in\sqgc(c)$ for some $1\leq c<2$, then  $\Gcal^{(c_1,c_2)}\in\sqgm(c)$.
\end{proposition}

\begin{proposition}
\label{ui3s7fg}
For every $d\in\Bbb{N}$, ${\cal S}_{d}\subseteq{\cal P}^{(1,d)}$. 
\end{proposition}

We visualise  the idea of the proof of \autoref{ui3s7fg} by some example, depicted in \autoref{asfgsdgasdfasdfsdfsadfsdf}. In \autoref{kil45io} we use the same idea for a more general result.
\autoref{asfgsdgasdfasdfsdfsadfsdf} motivates the definition of the $(c_1,c_2)$-extension of a graph class.
Intuitively, the fact that $H\in {\cal G}^{(c_{1},c_{2})}$ expresses the fact  that 
$H$ can be seen as a ``bounded'' distortion of a graph in ${\cal G}$ (after a fixed number of ``de-contractions'' and contractions).

 \begin{figure}[ht]
\begin{center}
 \scalebox{.9}{\includegraphics{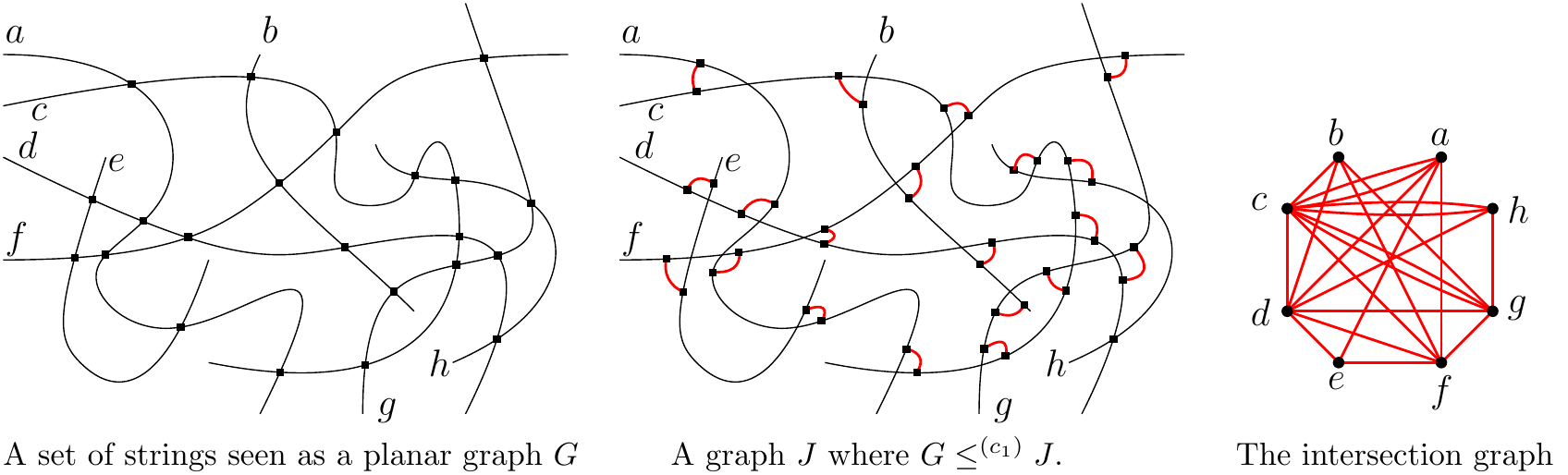}}
 \end{center}
\caption{An example of the proof of \autoref{ui3s7fg}. In the leftmost figure we 
we a collection of lines in the plane ${\cal L}=\{L_{1},\ldots,L_{k}\}$ whose intersection graph $G_{\cal L}$
is depicted in the rightmost figure and has maximum edge degree 9 because of line $c$
that meets other lines in 9 points, therefore $G_{\cal L}\in {\cal S}_{9}$. To see why $G_{\cal L}\in {\cal P}^{(1,9)}$, one may see the leftmost figure are a planar graph $P\in{\cal P}$ where vertices of degree 1 are discarded. The vertices of this planar graph can be seen as the result of the contraction of the red edges (seen as subgraphs of diameter 1) in the graph $J$ in the middle, i.e., $P\leq^{(1)} J$. Finally, the intersection graph $G_{\cal L}$ can be seen as a result of the contraction in $J$ of each one of the paths, on at most 9 vertices, to a single vertex. Therefore $G_{\cal L}\leq^{9} J$, hence $G_{\cal L}$ belongs 
in the $(1,9)$-extension of planar graphs. 
}
\label{asfgsdgasdfasdfsdfsadfsdf}
 \end{figure}

%
\autoref{jd9iodne},  combined with \autoref{j9iossnde}, provided the wider, so far, 
framework on the applicability of minor-bidimensionality: 
$\sqgm(1)$ contains ${\sf excl}(H)^{(c_1,c_2)}$ for every apex graph $H$ and positive integers $c_{1},c_{2}$.
As, by~\autoref{jd9iodne}, ${\cal P}\in \sqgc(1)$, ~\autoref{j9iossnde} and~\autoref{ui3s7fg} directly
classifies in  $\sqgm(1)$ the graph class  ${\cal S}_{d}$,
 and therefore a large family of bounded degree 
intersection graphs  (including ${\cal U}_{d}$ and ${\cal M}_{d}$). As a result of this, the 
applicability  of bidimensionality theory for minor-bidimensional problems has been extended to much wider families (not necessarily minor-closed) of graph classes of geometric nature~\cite{GrKoTh2014}.

\paragraph{Our main result.}


\begin{definition}[Intersection graphs]\label{interd}
Given a graph $G$ and a set $S\subseteq V(G)$ we say that $S$ is a {\em connected}
set of $G$  if $G[S]$ is a connected graph. We also define by ${\cal C}(G)$  the set of all connected subsets of $V(G)$. Given a ${\cal C}\subseteq  {\cal C}(G)$, we define the {\em intersection graph
of ${\cal C}$ in $G$}, denoted by 
$I_{G}({\cal C})$, as the graph whose vertex set is ${\cal C}$, where two vertices $C_{1}$ and $C_{2}$  of  $I_{G}({\cal C})$ 
are connected by an edge if $C_{1}\cap C_{2}\neq\emptyset$, and, moreover, the multiplicity of the edge $\{C_{1},C_{2}\}$ is equal to $|V(C_{1}\cap C_{2})|$.
Given a graph class ${\cal G}$ we define the following class of graphs $${\sf inter}({\cal G})=\bigcup_{G\in {\cal G}}\{I_{G}({\cal C})\mid {\cal C}\subseteq  {\cal C}(G)\}.$$
\end{definition}

 In other words, ${\sf inter}({\cal G})$ contains all the intersection graphs of the connected vertex subsets of each of the graphs in ${\cal G}$.
Given a $d\in\Bbb{N}$, we define  ${\sf inter}_{d}({\cal G})$ as the set of graphs in ${\sf inter}({\cal G})$ that have  edge-degree at most $d$. 

However, also the degree bound is maintained, as indicated by the following easy lemma.

\begin{lemma}
\label{sdlstring}
 For every $d\in\Bbb{N}$,   ${\cal S}_{d}\subseteq {\sf inter}_{d}({\cal P})\subseteq {\cal S}_{d'}$, for some $d'=O(d^2)$.
\end{lemma}
\begin{proof}[Proof (sketch).]
We deal with the less trivial statement that  ${\sf inter}_{d}({\cal P})\subseteq {\cal S}_{O(d^2)}$. For this, let $H\in {\sf inter}_{d}({\cal P})$ such that $H=I_{G}({\cal C})$ for some collection ${\cal C}$ of connected subsets of  $V(G)$, for some $G\in {\cal P}$.
We choose the planar graph $G$ so that $|V(G)|+|E(G)|$ is minimized. This means 
that for every $C\in {\cal C}$, $G[C]$ is a  tree on at most $2d$ vertices.
If we now replace each tree $G[C]$ by a string ``surrounding'' it is easy to observe that two such string cannot have more than $O(d^2)$ points in common.
%
%
%
%
%
\end{proof}


\medskip
 Observe that \autoref{j9iossnde} exhibits some apparent ``lack 
of symmetry'' as the assumption is ``qualitatively stronger'' than the conclusion. This  does not permit the application of bidimensionality for {\sl contraction}-bidimensional parameters on classes further than those of apex-minor free graphs. In other words, the results in~\cite{GrKoTh2014} covered, for the case of ${\cal S}_{d}$, the \green{green} triangles in Figure~\ref{bidipic} but left the \red{red} triangles open.
The main result of this paper is to  fill this gap by proving the following extension of \autoref{j9iossnde}. The main result of this paper is the following.

\begin{theorem}
\label{j9io3nj78}
Let $c_{1}$ and $c_{2}$ be two positive integers. If ${\cal G}\in\sqgc(c)$ for some $1\leq c<2$, then  $\Gcal^{(c_1,c_2)}\in\sqgc(c)$.
\end{theorem}

\paragraph{Consequences.}

We call a graph class {\em monotone} if it is closed under taking of subgraphs, i.e., 
every subgraph of a graph in ${\cal G}$ is also a graph in ${\cal G}$. A powerful consequence of \autoref{j9io3nj78} is the following (the proof is postponed in \autoref{psrdodof8r}).
\begin{theorem}
\label{this_oklok}
If ${\cal G}$ is a monotone graph class, where ${\cal G}\in\sqgc(c)$ for some $1\leq c<2$, and $d\in\Bbb{N}$, then ${\sf inter}_{d}({\cal G})\in\sqgc(c)$.
\end{theorem}

\begin{figure}
\begin{center}
\scalebox{1}{\includegraphics{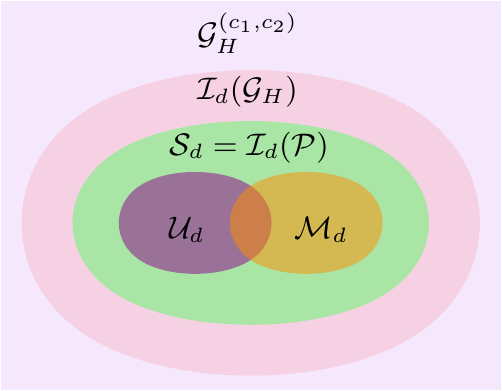}}
\end{center}
\caption{The hierarchy of graph classes where \autoref{j9iossnde} applies. ${\cal U}_{d}$ and ${\cal M}_{d}$ are the bounded-degree unit-disk and map graphs respectively (where the results of~\cite{FoLoSa2012,FominLS18exclu} apply).
${\cal S}_{d}$ are the bounded-degree string graphs, while ${\sf inter}_{d}({\sf excl}(H))$ are the bounded-degree  intersection graphs of connected sets of $H$-minor free graphs, where $H$ is an apex graph.}
\label{fislex}
\end{figure}

Combining~\autoref{jd9iodne} and~\autoref{this_oklok} we obtain that $\sqgc(1)$ contains ${\sf inter}_{d}({\sf excl}(H))$ for every apex graph $H$. 
This extends the applicability horizon of contraction-bidimensionality further than apex-minor free graphs (see \autoref{fislex}). 
As a (very) special case of this, we have that ${\cal S}_{d}\in\sqgc(1)$. Therefore, on ${\cal S}_{d}$, the results described in \autoref{resbidi} apply for contraction-bidimensional problems as well.

\medskip

This paper is organized as follows. In \autoref{defprel}, we give 
the necessary definitions  and some preliminary results.
We prove \autoref{this_oklok} in \autoref{psrdodof8r} while
\autoref{mainipd}
is dedicated to the proof of \autoref{j9io3nj78}. We should  
  stress that this proof
is quite different 
than the one of \autoref{j9iossnde} in~\cite{GrKoTh2014}. 
Finally, \autoref{copsed} contains some discussion and open problems.

\section{Definitions and preliminaries}\label{defprel}

We denote by $\Bbb{N}$ the set of all non-negative integers.
Given $r,q\in\Bbb{N},$ we define $[r,q]=\{r,\ldots,q\}$ and $[r]=[1,r]$.

All graphs in this paper are undirected, loop-less, and may have multiple edges. If a graph has no multiple edges, we call it {\em simple}. Given a graph $G$, we denote by $V(G)$ its vertex set and by $E(G)$ its edge set. 
Let $x$ be a vertex or an edge of a graph $G$ and likewise for $y$; their {\em distance} in $G$, denoted by ${\bf dist}_{G}(x,y)$, is the smallest number of vertices of a path in $G$ that contains them both. 
Moreover if $G$ is a graph and $x \in V(G)$, we denote by $N_G^c(x)$, for each $c \in \mathbb{N}$, the set $\{y \mid y \in V(G),\ {\bf dist}_{G}(x,y) \leq c+1\}$.
 For any set of vertices $S\subseteq V(G)$, we denote by $G[S]$ the subgraph of $G$ induced by the vertices from $S$. 
If $G[S]$ is connected, 
 then we say that $S$ is a \emph{connected vertex set} of $G$.
We define the \emph{diameter} of a connected subset $S$ as the maximum pairwise distance between any two vertices of $S$.  The {\em edge-degree} of a vertex $v\in V(G)$ is the number 
of edges that are incident to it (multi-edges contribute with their multiplicity to this number).

\begin{figure}[h]
\begin{center}
~~~~\scalebox{.14}{\includegraphics{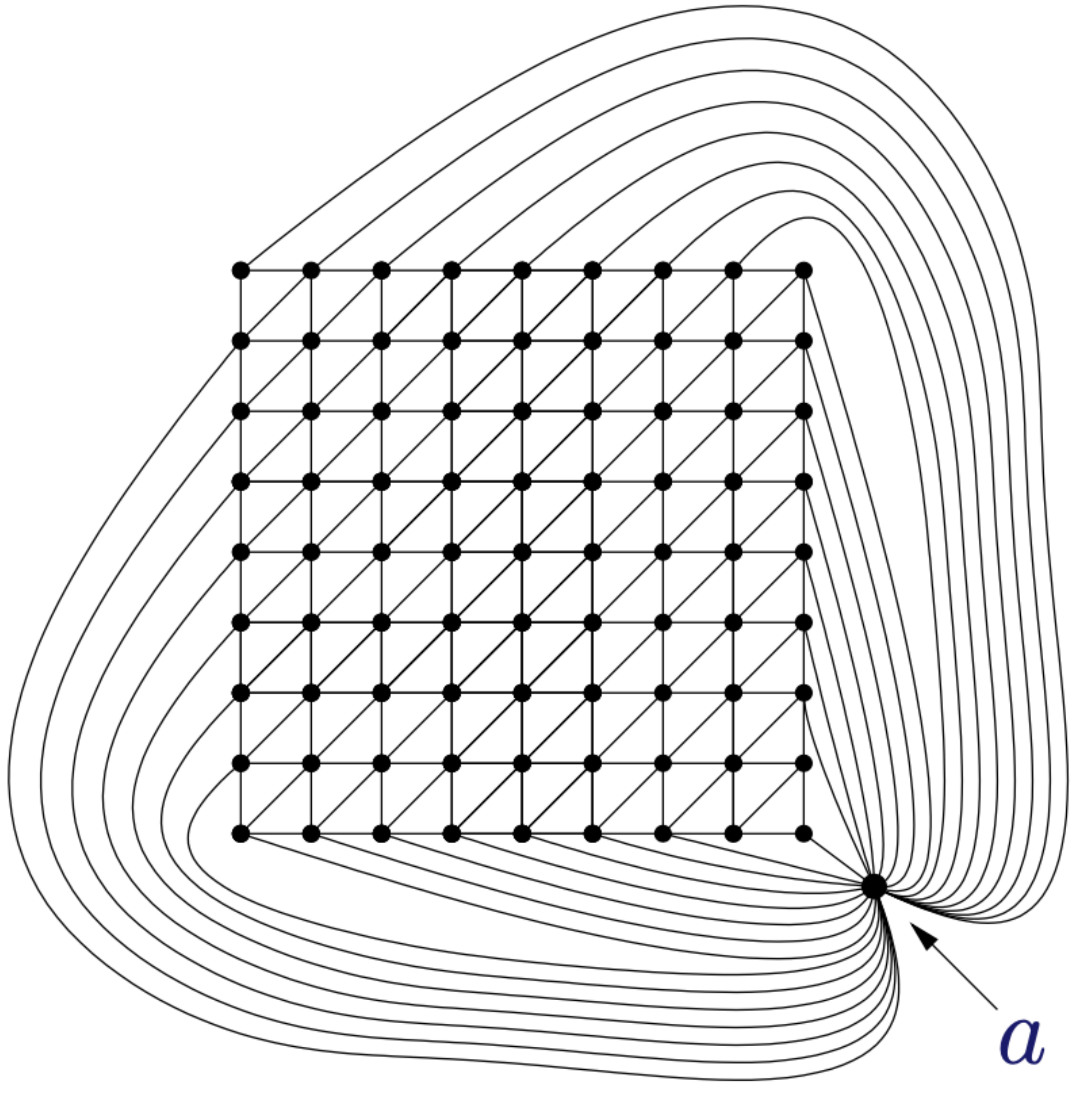}}
\end{center}
\caption{The graph $\hat{\Gamma}_{9}$.}
\label{fig-gamma-reg-hat}
\end{figure}

\begin{definition}[Grids]\label{grids}
The
\emph{$(k\times k)$-grid}, denoted by $\boxplus_k$, is the
graph whose vertex set is $[0,k-1]^2$ and two vertices $(i,j)$ and $(i',j')$ are adjacent if $|i-i'|+|j-j'|=1.$ 
For $k\geq 3$, the graph ${\Gamma}_k$ (resp. $\hat{\Gamma}_k$),
is   defined if we add in $\boxplus_k$
all the edges in  $\{\{(i+1,j),(i,j+1)\}\mid (i,j)\in [0,k-2]^{2}\}$ as well as all the edges 
between $(k-1,k-1)$ (resp. a new vertex $a$) and the vertices in $\{(i,j)\mid [0,k-1]^{2}\setminus [1,k-2]^{2}\}$ that have not been added already. 
For an example of ${\Gamma}_k$ (resp. $\hat{\Gamma}_k$), see \autoref{ddsdfgdfgdfg_sdfgdfgdfgdfg} (resp. \autoref{fig-gamma-reg-hat}).
Notice that ${\Gamma}_k$ is a triangulation of $\boxplus_k$.
In each of these graphs we 
denote the vertices of the underlying grid by their coordinates $(i,j)\in[0,k-1]^2$
agreeing that the upper-left corner (i.e., the unique vertex of degree 3) is the vertex $(0,0)$.
{$\hat{\Gamma}_k$ has two vertices of degree $3$, the top left and the bottom right of the grid part.} We call $\Gamma_{k}$ the {\em uniformly triangulated grid} and  $\hat{\Gamma}_{k}$ the {\em extended uniformly triangulated grid}. 
\end{definition}

\begin{definition}[Treewidth]
A {\em tree-decomposition} of a graph $G$, is a pair $(T,{\cal X})$, where $T$ is a tree and ${\cal X}=\{X_t:\ t\in V(T)\}$ is a family of subsets of $V(G)$, called {\em bags}, such that the following three properties are satisfied:

\noindent$\bullet$\ $\bigcup_{t\in V(T)}X_t=V(G)$,

\noindent$\bullet$\  for every edge $e\in E(G)$ there exists $t\in V(T)$ such that $e\subseteq X_t$, and

\noindent$\bullet$\  $\forall v\in V(G)$, the set $T_{v}=\{t\in V(T)\mid v\in X_{t}\}$ is a connected vertex set of  $T$.

The {\em width} of a tree-decomposition is the cardinality of the maximum size bag minus 1 and the {\em treewidth} of a graph $G$ is the minimum width {over all the} tree-decomposition{s} of $G$. We denote the treewidth of $G$ by $\tw(G)$.  
\end{definition}

\begin{lemma}
\label{lemma:twcont}
Let $G$ be a graph and let $H$  be a  $c$-size contraction of $G$.
Then $\tw(G)\leq (c+1)\cdot (\tw(H)+1)-1$.
\end{lemma}

\begin{proof}
By definition, since $H$ is a $c$-size contraction of $G$, there is a mapping between each vertex of $H$ and a connected set of at most $c$ edges in $G$, so that by contracting these edge sets we obtain $H$ from $G$. The endpoints of these edges form disjoint connected sets in $G$, implying a partition of the vertices of $G$ into connected sets $\{V_x\mid x\in V(H)\}$, where $|V_x|\leq c+1$ for any vertex $x\in V(H)$.

Consider now a tree decomposition  $(T,{\cal X})$ of $H$. We claim that the pair $(T,{\cal X'})$, where $X'_t:=\bigcup_{x\in X_t} V_x$ for $t\in T$ is a tree decomposition of $G$. Clearly all vertices of $G$ are included in some bag, since all vertices of $H$ did. Every edge of $G$ with both endpoints in the same part of the partition is in a bag, as each of these vertex sets is placed as a whole in the same bag. If $e$ is an edge of $G$ with endpoints in different parts of the partition, say $V_x$ and $V_y$, then this implies that $\{x,y\}\in E(H)$. Thus, there is a node $t$ of $T$ for which $x,y \in X_t$ and therefore $e\subseteq X_t'$. Moreover, the continuity property remains unaffected, since for any vertex $x \in V(H)$ each vertex in $V_x$ induces the same subtree in $T$ that $x$ did. 
\end{proof}

In \autoref{defsassult} we present all the notation that we use in this paper.
\begin{table}[htp]
{\small 
\begin{center}
\begin{tabular}{|c|c|c|}
\hline
\hline
Symbol& Combinatorial object & Definition
\\
\hline\hline
$\boxplus_k$ & $(k\times k)$-grid & \ref{grids}\\ 
\hline
$\Gamma_{k}$ & uniformly triangulated grid &  \ref{grids}\\ \hline
$\hat{\Gamma}_{k}$ & extended uniformly triangulated grid & \ref{grids} \\ \hline\hline
$I_{G}({\cal C})$ & the intersection graph
of a set ${\cal C}$ of connected subsets of the vertices of a graph  $G$ &\ref{interd}\\ \hline
$G_{\cal L}$ & the intersection graph of  a collection ${\cal L}=\{L_{1},\ldots,L_{k}\}$ of lines in the plane&\ref{interd}\\ \hline
\hline
${\cal P}$ & planar graphs&\ref{planall}\\ \hline
${\cal S}_{d}$ & $d$-bounded degree string graphs&\ref{stringgg}\\ \hline
${\cal M}_{d}$ & $d$-bounded degree map graphs&\ref{fnolki}\\ \hline
${\cal U}_{d}$ & $d$-bounded degree unit disk  graphs&\ref{fnolki}\\ \hline\hline 
${\sf excl}(H)$ & the class of graphs  excluding the graph $H$ as a minor&\ref{allmincon}\\ \hline\hline
${\sf diss}({\cal G})$ &  the graph class that is the  dissolution closure of the graph class ${\cal G}$& \ref{fifdiss}\\\hline
$\Gcal^{(c_1,c_2)}$ & the graph class that is the {\em $(c_1,c_2)$-extension} of a graph class ${\cal G}$&\ref{planall}\\ \hline\hline
\end{tabular}
\end{center}
\caption{Graphs, graph classes,  and functions of the paper.}
\label{defsassult}
}
\end{table}%

\section{Proof of \autoref{this_oklok}}
\label{psrdodof8r}

We start with the following useful property of the contraction relation.
We use $\delta(G)$ for the minimum number of edges that are incident to a vertex of the graph $G$. Given a vertex $v$ in $G$ incident to exactly two edges $e_1=\{v,x\}$ and $e_2=\{v,y\}$, the {\em dissolution of $v$} in $G$ is the operation of  removing  $e_1$ and $e_2$ from $G$ and then we add the edge $\{x,y\}$. If the graph $H$ occurs from $G$ after applying some (possibly empty) sequence of vertex dissolutions, then we say that  $H$ is a {\em dissolution} of $G$. We also say that $H$ is a {\em topological minor} of $G$ if $H$ the dissolution of some subgraph of $G$.

\begin{lemma}
\label{asfsvcxavcsdvvsd}
Let $Q$ be a graph where $\delta(Q)\geq 3$ and let $H,G$ be 
graphs where $H$ is a dissolution of $G$. If $Q\leq G$, then $Q\leq H$. 
\end{lemma}

\begin{proof}
 As $Q \leq G$ then there exist $\sigma: V(G) \to V(Q)$ such that for
all $x \in V(Q)$, $G[\sigma^{-1}(x)]$ is a non-empty graph and for
every $x,y\in V(Q),$ $\{x,y\}\in E(Q) \iff G[\sigma^{-1}(x)\cup
\sigma^{-1}(y)]$ is connected.

 Let $v$ be a vertex in $G$ incident to exactly two edges
 $e_1=\{v,v'\}$ and $e_2=\{v,v''\}$, and let $G'$ be the graph
 obtained from $G$ after the dissolution of $v$. Let $\sigma': V(G')
 \to V(Q)$ such that for all $z \in V(G')$, $\sigma'(z) = \sigma(z)$.
 As the dissolution maintains connectivity, we have that for every
 $x,y\in V(Q),$ $\{x,y\}\in E(Q) \iff G[\sigma'^{-1}(x)\cup
 \sigma'^{-1}(y)]$ is connected. Moreover, as  $\delta(Q)\geq 3$, we
 know that for each $x \in V(Q)$, there exists $z\in \sigma^{-1}(x)$
 such that $z$ has edge degree at least $3$. In particular we know
 that $z$ is different from $v$. So we have that $G[\sigma'^{-1}(x)]$
 is a non-empty graph. Thus $Q \leq G'$. The lemma follows by
 iterating this argument. \end{proof}

\begin{definition}[The function {\bf bcg}]
Given a graph $G$, we define ${\bf bcg}(G)$ as the maximum $k$ for which $G$ can be contracted to the uniformly triangulated grid $\Gamma_{k}$. 
\end{definition}
Notice that ${\bf bcg}$ is a contraction-closed parameter, i.e., if $H\leq G$, then ${\bf bcg}(H)
\leq {\bf bcg}(G)$.

\begin{lemma}
\label{fasdfsdfsd}
Let $H$ and $G$ be two graphs. If $H$ is a dissolution of $G$, then ${\bf bcg}(H)={\bf bcg}(G)$.
\end{lemma}

\begin{proof}
The fact that ${\bf bcg}(H)\leq {\bf bcg}(G)$ follows from the fact that $H$ is also a contraction 
of $G$ and taking into account the contraction-closedness of ${\bf bcg}$.
The fact that ${\bf bcg}(G)\leq {\bf bcg}(H)$ follows by taking into account that $\delta(\Gamma_{k})\geq 3$
and applying inductively \autoref{asfsvcxavcsdvvsd} to the vertices of degee 2 in $G$ that need to be dissolved in order to transform $G$ to $H$.
\end{proof}

\begin{definition}\label{fifdiss}
Given a graph class ${\cal G}$, we define the {\em dissolution closure} of ${\cal G}$ as the graph class ${\sf diss}({\cal G})$  containing 
all the dissolutions of the graphs in ${\cal G}$. 
\end{definition}

We observe the following.
\begin{lemma}
\label{asdfdfsdgfsgfdg}
If  ${\cal G}\in\sqgc(c)$ for some $1\leq c<2$, then  ${\sf diss}({\cal G})\in\sqgc(c)$.
\end{lemma}
\begin{proof}
Suppose that ${\cal G}\in\sqgc(c)$ for some $1\leq c<2$, wich 
implies that 
\begin{eqnarray}
\forall G \in \Gcal\   \tw(G)  \leq  \lambda\cdot  ({\bf bcg}(G))^c. \label{asdfsdf}
\end{eqnarray}
Let $H\in {\sf diss}({\cal G})$ and let $G\in {\cal G}$ such that  $H$ is a dissolution of  $G$.
By \autoref{fasdfsdfsd}, ${\bf bcg}(H)={\bf bcg}(G)$ and from~\eqref{asdfsdf},  $\tw(G)\leq \lambda ({\bf bcg}(H))^c$.  As $H$ is a minor of $G$, we have that $\tw(H)\leq \lambda ({\bf bcg}(H))^c$ and the lemma follows. 
\end{proof}

The next lemma uses as a departure point the same idea as the one of proof of \autoref{ui3s7fg}, visualized by the example of \autoref{asfgsdgasdfasdfsdfsadfsdf}.

\begin{lemma}
\label{kil45io}
If ${\cal G}$ is a graph class that is topological minor closed, then  ${\sf inter}_{d}({\cal G})\subseteq {\cal G}^{(d+1,d-1)}$.
\end{lemma}

\begin{proof}
 Let $H$ be a graph on $h$ vertices in  ${\sf inter}_{d}({\cal G}),$ for some $d\in \Bbb{N}$.
This means that there is a graph $G$ in ${\cal G}$ such that 
we can see the vertices of $H$ as a set ${\cal C}=\{C_{1},\ldots,C_{h}\}$ of connected subsets of $G$ and each multi-edge $e=\{C_{i},C_{j}\}$ of $H$ corresponds
to two mutually intersecting subsets of ${\cal C}$ and  the multiplicity 
of $e$ is $|C_{i}\cap C_{j}|$.
For every $\{i,j\}\in{h\choose 2}$, we set $V_{i,j}=C_{i}\cap C_{j}$, $m_{i,j}=|C_{i}\cap C_{j}|$, and we assume  that $e_{i,j}=\{C_{i},C_{j}\}$ is a multi-edge of $H$ 
of multiplicity $m_{i,j}$ (if this edge does not exists in $H$, then the multiplicity of $e_{i,j}$ is 0).

We  define  $V_{i}=\bigcup_{j\in [h]}V_{i,j}$, 
for every $i\in[h]$ and also set ${V}=\bigcup_{i\in[h]} V_{i}$.
Notice that, for each $i \in [h]$, $|V_{i}|$ is upper bounded by the edge-degree, in $H$, of the vertex $C_{i}$, therefore,
$|V_{i}|\leq d$ for each $i\in[h]$. Also, a vertex in ${V}$ cannot belong in more 
that $d+1$ distinct $C_{i}$’s as, otherwise $H$ would contain a clique with at least $d+2$ vertices.
As $H \in {\sf inter}_{d}({\cal G})$, this is not possible.

Recall that, for each $i \in [h]$, $V_{i}$ is a subset of $C_{i}$ and let $T_{i}$ be a minimum-size tree of 
$G[C_{i}]$ containing the vertices of $V_{i}$.
We partition the set of vertices of $T_{i}$ into three sets  $V_{i},\overline{V}_{i},D_{i}$
where among the vertices in $V(T_i)\setminus V_{i}$, $D_{i}$ are the vertices of degree 2 and $\overline{V}_{i}$ are the rest. By minimality, the leaves of  $T_{i}$ belong in $V_{i}$. Moreover,  there is no vertex in $\overline{V}_{i}$ that belongs to some other $\overline{V}_{i’}$, $i \in [h] \setminus  \{i\}$.
 We denote by $\hat{T}_{i}$ the tree obtained from $T_{i}$ if we dissolve in 
$T_{i}$ all vertices of $D_{i}$. That way we can still partition the vertices of each $\hat{T}_i$, $i \in [h]$,
into  $V_{i}$ and $\overline{V}_{i}$.
Also, it is easy to see that $\hat{T}_{i}$ has diameter at most $|V_{i}|-1\leq d-1$. 

We define the graph  $G’:=\bigcup_{i\in[h]}\hat{T}_{i}$.
Notice that $G'$  is obtained from $\bigcup_{i\in[h]} T_{i}$ (that is a subgraph of $G$) after we dissolve all vertices in $\bigcup_{i\in[h]}D_{i}$.
Therefore $G'$ is a topological minor of $G$, thus $G'\in{\cal G}$.
We consider the collection ${\cal T}=\{\hat{T}_{1},\ldots,\hat{T}_{h}\}$
of connected subgraphs of $G'$.

We define the graph $J$ to be the disjoint union of the $h$  trees in ${\cal T}$ in which, for each $x \in {V}$, we add a clique between all the copies of $x$.
Notice that each added clique has size at least $2$ and at most $d+1$.

Observe now that $G'\leq^{(d+1)} J$, as $G'$ is obtained after contracting in $J$ the aforementioned pairwise disjoint cliques. Moreover, $H\leq^{d-1} J$ as $H$ is obtained 
after we contract in $J$ each $\hat{T}_{i}$ (of diameter $\leq d-1$) to a single vertex. 
As $G'\in {\cal G}$, we conclude that $H\in {\cal G}^{(d+1,d-1)}$ as required.
\end{proof}

We are now ready to prove \autoref{this_oklok}.

\begin{proof}[Proof of \autoref{this_oklok}]
Let ${\cal G}$ be a monotone graph class in $\sqgc(c)$ for some $1\leq c<2$
and let ${\cal D}={\sf diss}({\cal G})$.
From \autoref{asdfdfsdgfsgfdg}, ${\cal D}\in\sqgc(c)$
and by the monotonicity of ${\cal G}$, we have that ${\sf diss}({\cal G})$ is closed under taking of topological minors. Therefore, from \autoref{kil45io},
 ${\sf inter}_{d}({\cal D})\subseteq {\cal D}^{(d+1,d-1)}$ and from \autoref{j9io3nj78}, ${\sf inter}_{d}({\cal D})\in \sqgc(c)$. The result follows because ${\cal G}\subseteq {\cal D}$, as  this implies that 
 ${\sf inter}_{d}({\cal G})\subseteq {\sf inter}_{d}({\cal D})$.
\end{proof}

\section{Proof of \autoref{j9io3nj78}}
\label{mainipd}

Let $H$ and $G$ be graphs and $c$ be a non-negative integer. 
If  $H\leq_{\sigma}^{c}G$, then we say that $H$ is a {\em $\sigma$-contraction} of $G$, and denote this by $H\leq_\sigma  G$.

 Before we proceed 
the the proof  of \autoref{j9io3nj78} we make first the following three observations.
 (In all statements, we  assume that $G$ and $H$ are two graphs and $\sigma: V(G) \to V(H)$ such that $H$ is a $\sigma$-contraction of $G$.)

\begin{observation}
\label{obs:1}
Let $S$ be a connected subset of $V(H)$.
Then the set $\bigcup_{x\in S}\sigma^{-1}(x)$ is connected in $G$.
\end{observation}
\begin{observation}
\label{obs:2}
Let $S_{1}\subseteq S_{2}\subseteq V(H)$. Then 
$\sigma^{-1}(S_{1})\subseteq \sigma^{-1}(S_{2})\subseteq V(G)$.
\end{observation}

\begin{observation}
\label{obs:3}
Let $S$ be a connected subset of $V(G)$.
Then the diameter of $\sigma(S)$ in $H$ is at most the diameter of $S$ in $G$.
\end{observation}

Given a graph $G$ and $S_{1},S_{2}\subseteq V(G)$ we say that $S_{1}$ and $S_{2}$ \emph{touch} if either $S_{1}\cap S_{2}\neq\emptyset$
or there is an edge of $G$ with one endpoint in $S_{1}$ and the other in $S_{2}$.

We say that a collection ${\cal R}$ of paths of a graph is  {\em internally disjoint}
if none of the internal vertices, i.e., none of the vertex of degree $2$,
 of some path in ${\cal R}$  is a vertex of some other path in ${\cal R}$. 
Let ${\cal A}$ be a collection of subsets of $V(G)$. We say that ${\cal A}$ 
is a {\em connected packing} of $G$ if its elements are connected and pairwise disjoint.
If additionally ${\cal A}$ is a partition of $V(G)$, then we say that 
${\cal A}$ is a {\em connected partition} of $G$ and if, additionally, all its elements have diameter bounded 
by some integer  $c$, then we say that ${\cal A}$
is a {\em $c$-diameter partition} of $G$.

\subsection{$\Lambda$-state configurations.} 

\begin{definition}($\Lambda$-state configurations)
{Let $G$ be a graph.}
Let $\Lambda=({\cal W},{\cal E})$ be a graph whose vertex 
set is a connected packing of $G$, i.e., its vertices are connected subsets of $V(G)$.
A {\em $\Lambda$-state configuration}  of a graph $G$ is a quadruple  ${\cal S}=({\cal X},\alpha,{\cal R},\beta)$ where 
\begin{enumerate}
\item ${\cal X}$ is a  connected packing of $G$, 
\item $\alpha$ is a bijection from ${\cal W}$ to 
${\cal X}$ such that for every $W\in {\cal W}$, $W\subseteq \alpha(W)$,
\item ${\cal R}$ is a collection of internally disjoint paths of $G$, and 
\item $\beta$ is a bijection from ${\cal E}$ to ${\cal R}$ such that 
if $\{W_{1},W_{2}\}\in {\cal E}$ then the endpoints of $\beta(\{W_{1},W_{2}\})$ are in $W_{1}$ and $W_{2}$
and $V(\beta(\{W_{1},W_{2}\}))\subseteq \alpha(W_{1})\cup \alpha(W_{2})$.
\end{enumerate}

\end{definition}

\begin{definition}(States, freeways, clouds, and coverage)
A $\Lambda$-state configuration ${\cal S}=({\cal X},\alpha,{\cal R},\beta)$ of $G$ is  {\em complete} if ${\cal X}$ is a partition of $V(G)$.
We refer to the elements of ${\cal X}$ as the {\em states} of ${\cal S}$ and to the elements 
of ${\cal R}$ as the {\em freeways}  of ${\cal S}$.
We define ${\sf indep}({\cal S}) = V(G) \sm \bigcup_{X \in \cal X} X.$ Note that if $\cal S$ is a $\Lambda$-state configuration of $G$, $\cal S$ is complete if and only if ${\sf indep}({\cal S}) = \es$.

Let ${\cal A}$ be a $c$-diameter partition of $G$.  We refer to the sets of ${\cal A}$  as the 
{\em  ${\cal A}$-clouds} of $G$. 
We define ${\sf front}_{\cal A}({\cal S})$ as the set of all ${\cal A}$-clouds of $G$ that are not subsets of some 
$X\in\Xcal$.
Given a ${\cal A}$-cloud $C$ and a state $X$ of ${\cal S}$ we say that $C$ {\em shadows} $X$ 
if $C\cap X\neq \emptyset$. 
The {\em coverage} ${\sf cov}_{\cal S}(C)$ of an ${\cal A}$-cloud $C$ of $G$ is the number of states of ${\cal S}$ that are shadowed by $C$. 
A $\Lambda$-state configuration ${\cal S}=({\cal X},\alpha,{\cal R},\beta)$ of  $G$ is {\em ${\cal A}$-normal} if its satisfies the following conditions:
\begin{itemize}
\item[(A)] If a  ${\cal A}$-cloud $C$ intersects some $W\in {\cal W}$, then $C\subseteq \alpha(W)$.
\item[(B)] If a ${\cal A}$-cloud over ${\cal S}$ intersects the vertex set of at least 
two freeways of  ${\cal S}$, then it shadows at most one state of  ${\cal  S}$.
\end{itemize}
We define ${\sf cost}_{\cal A}({\cal S}) = \sum_{C \in {\sf front}_{\cal A}({\cal S})} {\sf cov}_{\cal S}(C).$
Given  $S_{1}\subseteq S_{2}\subseteq V(G)$ where $S_{1}$ is connected,
we define ${\sf cc}_{G}(S_{2},S_{1})$ as the (unique) connected component 
of $G[S_{2}]$ that contains $S_{1}$.

\end{definition}

\subsection{Triangulated grids inside triangulated grids}

The next lemma is the main combinatorial engine of our results.
We assume that $H\leq^{c} G$ and  $\Gamma_{k}\leq G$. Here $H$ should be seen as the result of a ``shrink'' of $G$ in the sense 
that $G$ can be contracted to $H$ so that each vertex of $H$ is created after a bounded number of contractions. The lemma states that if $G$ can be contracted to a uniformly triangulated grid, then   $H$,  as a ``shrunk version'' of $G$, can also be contracted to a uniformly triangulated grid that is no less than linearly smaller.

The proof strategy views the graph $G$ as being contracted into a uniformly triangulated grid $\Gamma_{k}$ (see \autoref{wergdsfgdfgdfgdf}), we choose a scattered set of
``capitals'' in it (the black vertices in \autoref{wergdsfgdfgdfgdf}). Then 
we set up a ``conquest'' procedure where each capital is trying to expand to a country. This procedure has three phases.
The first phase is the {\sl expansion face} where each country tries to incorporate unconquested territories around it (the limits of this expansion  is depicted by the red cycles in \autoref{wergdsfgdfgdfgdf}).
The second phase is the  {\sl clash face}, where different countries are fighting for disputed territories. Finally, the third phase is the {\sl annex phase}, where each country naturally incorporates remaining enclaves.
The end of this war creates a set of countries occupying the whole $G$ that, when contracted, give rise to a uniformly triangulated grid $\Gamma_{k'}$, where $k'=\Omega(k)$.

%
%

\begin{lemma}
\label{lemma_4}
Let $G$ and $H$ be graphs and $c,k$ be non-negative integers such that $H\leq^{c} G$ and $\Gamma_{k}\leq G$. 
Then $\Gamma_{k'}\leq H$ where $k'=\lfloor\frac{k-1}{2c+1}\rfloor-1$.
\end{lemma}

\begin{proof} Let $k^*=1+(2c+1)\cdot (k'+1)$ and observe that $k^{*}\leq k$, therefore
 $\Gamma_{k^{*}}\leq \Gamma_{k}\leq G$. For simplicity we use $\Gamma=\Gamma_{k^{*}}$.
Let $\phi: V(G)\rightarrow V(H)$ such that $H\leq^{c}_{\phi} G$
and let { $\sigma: V(G)\rightarrow V(\Gamma)$}
such that {$\Gamma\leq_{\sigma} G$.}
We define ${\cal A}=\{\phi^{-1}(a)\mid a\in V(H)\}$. Notice that ${\cal A}$ is a $c$-diameter partition of $G$.

%
For each $(i,j)\in\intv{0}{k'+1}^{2},$ we define $b_{i,j}$ to be the vertex of ${\Gamma}$  with coordinate $(i(2c+1),j(2c+1))$. 
We set $Q_{\rm in} = \{b_{i,j}\mid (i,j)\in\intv{1}{k'}^2\}$ and $Q_{\rm out}=\{b_{i,j}\mid (i,j)\in\intv{0}{k'+1}^2\}\setminus Q_{\rm in}$. Let also 
$Q=Q_{\rm in}\cup \{b_{\rm out}\}$ were $b_{\rm out}$ is a new element
that does not belong in $Q_{\rm in}$. Here $b_{\rm out}$ can be seen as a vertex that 
``represents'' all vertices in $Q_{\rm out}$.

Let $q,p$ be two different elements of $Q$.
We say that $q$ and $p$ are {\em linked} if they both belong in $Q_{\rm in}$ and 
their distance in ${\Gamma}$ is  $2c+1$ or one of them is $b_{\rm out}$ and 
the other 
is $b_{i,j}$ where $i\in\{1,k'\}$ or $j\in\{1,k'\}$. 
%

\begin{figure}[ht]
\begin{center}
\scalebox{.3}{\includegraphics{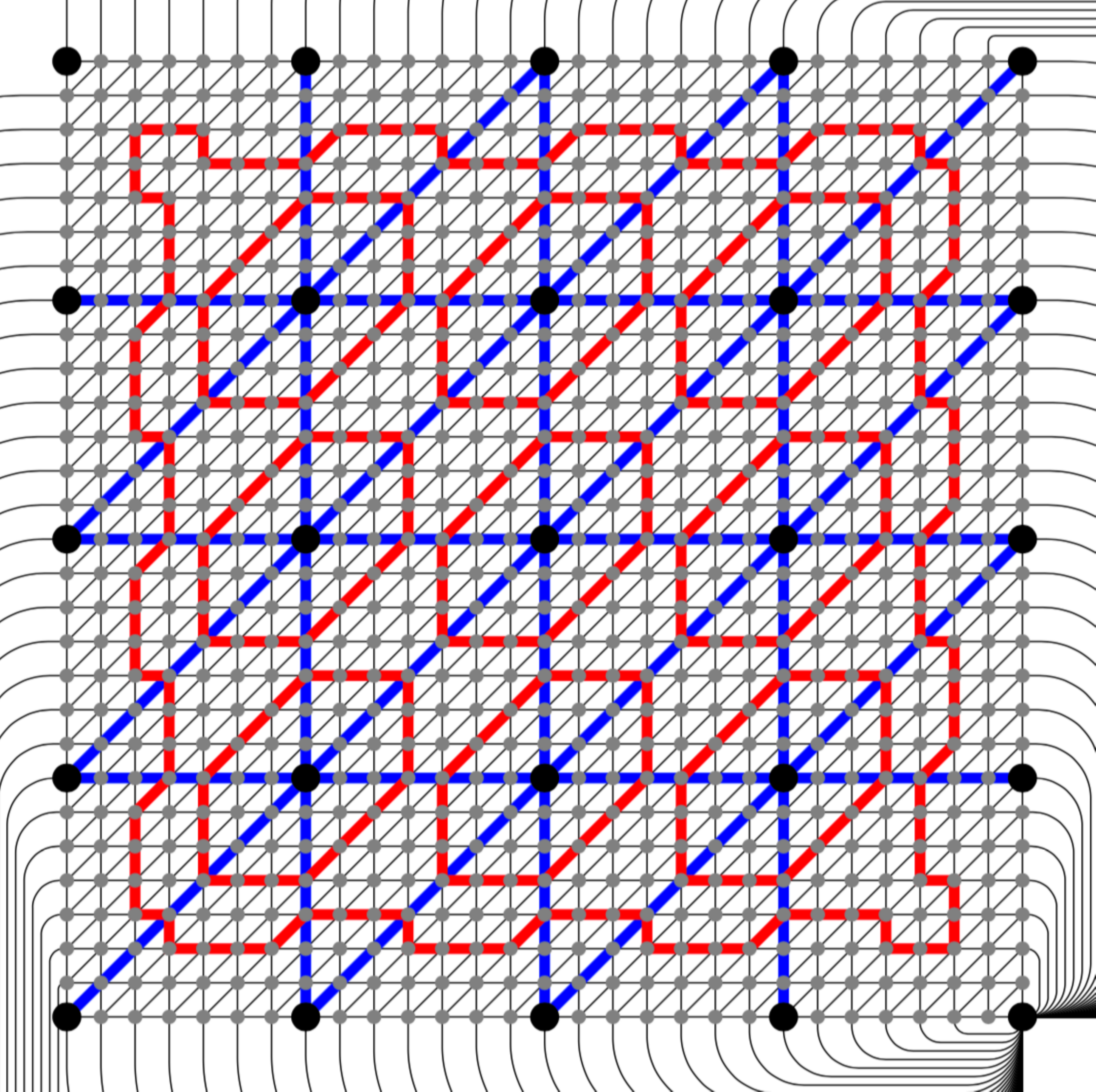}}
\end{center}
\caption{A visualization of the proof of \autoref{lemma_4}. In this whole graph $\Gamma_k$, we initialize our reaserch of $\hat{\Gamma}_{k'}$ such that every internal red hexagon will become a vertex of $\hat{\Gamma}_{k'}$ and correspond to a state and the border, also circle by a red line will become the vertex $b_{\rm out}$. The blue edges correspond to the freeways. Red  cycles correspond to the boundaries of the starting countries. Blue paths between big-black vertices are the freeways. Big-black vertices are the capitals.}
\vskip0cm
\label{wergdsfgdfgdfgdf}
\end{figure}
For each  $q\in Q_{\rm in}$, we define $W_{q}=\sigma^{-1}(q)$. 
$W_q$ is connected by the definition of  $\sigma$.
In case $q=b_{\rm out}$ we define $W_{q}=\bigcup_{q'\in Q_{\rm out}}\sigma^{-1}(q')$.
Note that as $Q_{\rm out}$ is a connected set of $\Gamma$, then, by Observation~\ref{obs:1}, $W_{b_{\rm out}}$ is connected in $G$.
We also define ${\cal W}=\{W_{q}\mid q\in Q\}$.
Given some $q\in Q$ we call $W_{q}$ the  {\em $q$-capital} of $G$ and 
a subset $S$ of $V(G)$ is a {\em capital} of $G$ 
if it is the { $q$-capital} for some $q\in Q$.
Notice that ${\cal W}$ is a connected packing of $V(G)$.

Let $q\in Q$. If $q\in Q_{\rm in}$ then we set $N_{q}=N^{c}_{{\Gamma}}(q)$. 
If 
$q=b_{\rm out}$, then we set $N_{q}=\bigcup_{q'\in  Q_{\rm out}} N^{c}_{{\Gamma}}(q')$.
Note that for every $q \in Q$, $N_q \subseteq V(\Gamma)$.
For every $q \in Q$, we define $X_q = \sigma^{-1}(N_q)$.
Note that $X_q \subseteq V(G)$.
We also set ${\cal X}=\{X_{q}\mid q\in Q\}$. 
Let $q$ and $p$ we two linked elements of $Q$.
If both $q$ and $p$ belong to $Q_{\rm in}$, and therefore are vertices of ${\Gamma}$,
then we define $Z_{p,q}$  as 
the unique shortest path between them in ${\Gamma}$. 
If $p=b_{\rm out}$ and $q\in Q_{\rm in}$, then we know that $q=b_{i,j}$ where  $i\in\{1,k'\}$ or $j\in\{1,k'\}$. 
In this case we define $Z_{p,q}$ as any shortest path in ${\Gamma}$
between $b_{i,j}$ and the vertices in $Q_{\rm out}$.
In both cases, we define $P_{p,q}$ 
by picking some path between $W_p$ and $W_q$  in $G[\sigma^{-1}(V(Z_{p,q}))]$
such that $|V(P_{p,q})\cap W_{q}|=1$ and $|V(P_{p,q})\cap W_{p}|=1$.

Let ${\cal E}=\{\{W_{p},W_{q}\}\mid \mbox{~$p$ and $q$ are linked}\}$
and let $\Lambda=({\cal W},{\cal E})$. Notice that $\Lambda$ 
is isomorphic to $\hat{\Gamma}_{k'}$ and consider the isomorphism 
that correspond each vertex $q=b_{i,j}$, $i,j \in \intv{1}{k'}^2$ to the vertex with coordinates $(i,j)$.
Moreover $b_{\rm out}$ corresponds to the apex vertex of $\hat{\Gamma}_{k'}$.

Let $\alpha: {\cal W}\rightarrow {\cal X}$ 
such that for every $q\in Q$, $\alpha(W_{q})=X_{q}$. Let 
also ${\cal R}=\{P_{p,q}\mid p,q \in Q, \mbox{~$p$ and $q$ are linked}\}$.
We define $\beta: {\cal E}\rightarrow {\cal R}$
such that if $q$ and $p$ are linked, then $\beta(W_{q},W_{p})=P_{p,q}$. We use notation 
${\cal S}=({\cal X},\alpha,{\cal R},\beta)$.

\begin{claim}
\label{claim_1}
${\cal S}$ is an ${\cal A}$-normal  $\Lambda$-state configuration of $G$. 
\end{claim}

\begin{myproof}[Proof of Claim~\ref{claim_1}.]
We first see that  $\cal S$ is a $\Lambda$-state configuration of $G$. 
Condition 1 follows by the definition of $X_{q}$ and Observation~\ref{obs:1}.
Condition 2 follows directly by the definitions of $W_{q}$ and $X_{q}$.
For Condition 3, we first observe that, by the construction of ${\Gamma}$ and 
the definition of $Z_{p,q}$, for any two pairs $p,q$ and $p',q'$ of pairwise 
linked elements of $Q$,  the paths $Z_{p,q}$ and $Z_{p',q'}$ are internally vertex disjoined
paths of ${\Gamma}$. 
It implies that $P_{p,q}$ and $P_{p',q'}$ can intersect each other only on the vertices of $W_p \cup W_q \cup W_{p'} \cup W_{q'}$.
But $P_{p,q}$ (resp. $P_{p',q'}$), by construction contains only two vertices of $W_p \cup W_q \cup W_{p'} \cup W_{q'}$ that are the extremities of $P_{p,q}$, (resp. $P_{p',q'}$).
So $P_{p,q}$ and $P_{p',q'}$ are internally vertex disjoined, as required. 
For Condition 4, assume that $\{W_{p},W_{q}\}\in{\cal E}$.
The fact that the endpoints of $\beta(\{W_{p},W_{q}\})$
are in $W_{p}$ and $W_{q}$  follows directly by the definition of $\beta(\{W_{p},W_{q}\})=P_{p,q}$.
It remains to prove that $V(\beta(\{W_{p},W_{q}\}))\subseteq \alpha(W_{p})\cup \alpha(W_{q})$
or equivalently, that  $V(P_{p,q})\subseteq X_{p}\cup X_{q}$. Observe that, if both $p,q\in Q_{\rm in}$, then
every vertex in the shortest path $Z_{p,q}$ should be within distance $c$ from either $p$ or $q$.
Similarly, if $p\in Q_{\rm in}$ and $q=b_{\rm out}$, then  every vertex 
in the shortest path $Z_{p,q}$ should be within distance $c$ from either $p$ or some vertex in $Q_{\rm out}$.
So for every $p,q \in Q$, with $p \not = q$, $Z_{p,q} \subseteq N_p \cup N_q$.
By Observation~\ref{obs:2}, every vertex in $\sigma^{-1}(V(Z_{p,q}))$ belongs 
 to $X_{p} \cup X_{q}$ and the required follows as $V(P_{p,q})\subseteq \sigma^{-1}(V(Z_{p,q}))$.
This completes the proof that ${\cal S}$ is a $\Lambda$-state configuration of $G$.

We now prove that ${\cal S}$ is ${\cal A}$-normal.
Recall that $\cal A$ be a c-diameter partition of $G$. 
Let $C$ be a $\cal A$-cloud and let $C' = \sigma(C)$ be a subset of $V({\Gamma})$. 
As $C$ is of diameter at most $c$, then, from Observation~\ref{obs:3}, $C'$
is also  of diameter at most $c$. Notice that if $C$ intersects some member $W$ of ${\cal W}$, then 
$C'= \sigma(C)$ also intersects  $\sigma(W)$, therefore $C'$ intersects  some element of 
$Q_{\rm in}\cup Q_{\rm out}$.
Assume $C'$ contains $p \in Q_{\rm in}\cup Q_{\rm out}$, then $C' \subseteq N_p$.
From Observation~\ref{obs:1},
$C\subseteq  X_{p}=\alpha(W_{p})$, therefore $C$ satisfies Condition (A).

%

By construction, the distance in ${\Gamma}$ between two elements of $Q_{\rm in}$  is  
either $2c+1$ or at least $4c+2$.
The distance in ${\Gamma}$ between on elements of $Q_{\rm in}$ and any element of $Q_{\rm out}$ is a multiple of $2c+1$.  
This implies that if $p,q\in Q$, $p \not = q$, $N_p \cap C' \not = \es$, and  $ N_q \cap C' \not = \es$, then $p$ and $q$ are linked. 
%

By construction, if $p$ and $q$ are linked, then for every $r \in Q$ and every $u \in Z_{p,q}$, ${\bf dist}_{\Gamma}(r,u) \geq \min ({\bf dist}_{\Gamma}(r,p), {\bf dist}_{G}(r,q))$, where for every $x\in Q_{\rm in}$, 
the quantity ${\bf dist}_{{\Gamma}}(x,b_{\rm out})$ is interpreted
 as $\min\{{\bf dist}_{\Gamma}(x,q')\mid {q' \in Q_{\rm out}}\}$. 
This implies that if $C'$ intersects $Z_{p,q}$ for some $p,q \in Q$, then for every $r\in Q\setminus\{p,q\}$, then
 $C'$ does not intersect $N_r$. We will use  this fact in the next paragraph towards completing the proof of Condition (B).



We now claim that if $C'$  intersects two distinct paths in 
$\{Z_{p,q}\mid (p,q)\in Q^{2}, p\neq q\}$, then $C'$ intersects 
at most one of the sets in $\{N_{q'}\mid q'\in Q\}$.
Let $Z_{p,q}$ and $Z_{p',q'}$
be two distinct paths intersected by $C'$. We argue first that 
$p,q,p',q'$ cannot be all different. Indeed, if this is the case, 
as $C'$ intersects $Z_{p,q}$ then $C'$ cannot intersect $N_{p'}$ or $N_{q'}$ as $p',q' \not \in \{p,q\}$.
As $Z_{p',q'} \subseteq N_{q'} \cup N_{p'}$, we have a contradiction.
Assume now that $p = p'$ and $q \not = q'$.
As $C'$ intersects $Z_{p,q}$, then it does not intersect $N_{r}$ for any $r \in Q \sm \{p,q\}$, and as it intersects $Z_{p,q'}$, then it does not intersect $N_r$ for any $r \in Q \sm \{p,q'\}$.
We obtain that $C'$ intersects at most one of the sets in $\{N_{r}\mid r\in Q\}$ that is $N_p$.
By definition of the states, we obtain that $C$ shadows at most one state that is $X_p$.
That completes the proof of condition (B). 
\end{myproof}

We define bellow three ways to transform a  $\Lambda$-state configuration of $G$. In each of them,
${\cal S}=({\cal X},\alpha,{\cal R},\beta)$ is an ${\cal A}$-normal  $\Lambda$-state configuration of $G$  and  $C$ is an  ${\cal A}$-cloud in ${\sf front}_{\cal A}({\cal S})$.

\begin{itemize}
\item[1.] The {\em expansion procedure}
applies when $C$ intersects at least two freeways of ${\cal S}$.
Let $X$ be  the  state of ${\cal S}$  shadowed by $C$ (this state is unique because 
of property (B) of {$\Acal$}-normality).  
We define $({\cal X}',\alpha',{\cal R}',\beta')={\sf expand}({\cal S},C)$ 
such that 
\begin{itemize}
\item ${\cal X}'={\cal X}\setminus \{X\}\cup \{X\cup C\}$, 
\item 
for each $W\in{\cal W}$, $\alpha'(W)=X'$ where $X'$ is the unique set of ${\cal X}'$ such that $W\subseteq X'$,
\item ${\cal R}'={\cal R}$, and  $\beta'=\beta$.
\end{itemize}


\item[2.]  The {\em  clash procedure} applies when $C$ intersects {\sl exactly} one freeway $P$ of ${\cal S}$.
Let $X_{1},X_{2}$ be the two  states of ${\cal S}$ that intersect this freeway.
 Notice that $P=\beta(\alpha^{-1}(X_{1}),\alpha^{-1}(X_{2}))$, as it is the only freeway with vertices in $X_1$ and $X_2$. 
Assume that  $(C \cap V(P)) \cap X_1 \not = \es$ 
(if, not, then swap the roles of $X_{1}$ and $X_{2}$). 
We define $({\cal X}',\alpha',{\cal R}',\beta')={\sf clash}({\cal S},C)$ as follows:
\begin{itemize}
\item ${\cal X}'=\{X_1\cup C\}\cup\bigcup_{X\in{\cal X}\setminus\{X_{1}\}}\{{\sf cc}_{G}(X\setminus C,\alpha^{-1}(X))\}$ (notice that $\alpha^{-1}(X)\subseteq X\setminus C$, for every $X\in{\cal X}$, because of property (A) of ${\cal A}$-normality), 
\item 
for each $W\in{\cal W}$, $\alpha'(W)=X'$ where $X'$ is the unique set of ${\cal X}'$ such that $W\subseteq X'$,
\item ${\cal R}'={\cal R}\setminus \{P\}\cup \{P'\}$, where $P'=P_1\cup P^{*}\cup P_{2}$ 
is defined as follows: let $s_{i}$ be the first vertex 
of $C$ that we meet while traversing $P$ when starting from its  endpoint that belongs in $W_{i}$
and  let  $P_{i}$ the subpath of $P$ that we traversed that way, for $i\in\{1,2\}$. We define $P^*$ by taking any path 
between $s_{1}$ and $s_{2}$ inside $G[C]$, and
\item $\beta'=\beta\setminus \{(\{W_{1},W_{2}\},P)\}\cup\{\{W_{1},W_{2}\},P'\}$.
\end{itemize}

\item[3:] The {\em annex procedure} applies when $C$ intersects  no freeway of ${\cal S}$
and {\sl touches} some country $X\in{\cal X}$.
We define $({\cal X}',\alpha',{\cal R}',\beta')={\sf anex}({\cal S},C)$
such that 
\begin{itemize}
\item ${\cal X}'=\{X_1\cup C\}\cup\bigcup_{X\in{\cal X}\setminus\{X_{1}\}}\{{\sf cc}_{G}(X\setminus C,\alpha^{-1}(X))\}$ (notice that $\alpha^{-1}(X)\subseteq X\setminus C$, for every $X\in{\cal X}$, because of property (A) of ${\cal A}$-normality), 
\item 
for each $W\in{\cal W}$, $\alpha'(W)=X'$ where $X'$ is the unique set of ${\cal X}'$ such that $W\subseteq X'$,
\item ${\cal R}'={\cal R}$, and  $\beta'=\beta$.
\end{itemize}\end{itemize}

\begin{claim}
\label{d0oi3d}
Let ${\cal S}=({\cal X},\alpha,{\cal R},\beta)$ be an ${\cal A}$-normal  $\Lambda$-state configuration of $G$,  and  $C\in{\sf front}_{\cal A}({\cal S})$.
Let ${\cal S}'={\sf action}({\cal S},C)$ where ${\sf action}\in\{{\sf expand}, {\sf clash}, {\sf anex}\}$.
Then ${\cal S}'$ is  an ${\cal A}$-normal  $\Lambda$-state configuration of $G$ where 
${\sf cost}({\cal S'},{\cal A}) \leq  {\sf cost}({\cal S},{\cal A})$.
Moreover,  if ${\sf cov}_{\cal S}(C)\geq 1$, then ${\sf cost}({\cal S'},{\cal A}) < {\sf cost}({\cal S},{\cal A})$
and if ${\sf cov}_{\cal S}(C)=0$ (which may be the case only when ${\sf action}={\sf anex}$), then
$|{\sf indep}({\cal S'})| < |{\sf indep}({\cal S})|$.
\end{claim}

\begin{myproof}[Proof of Claim~\ref{d0oi3d}.]

We first show that $\Scal'$ is an ${\cal A}$-normal  $\Lambda$-state configuration of $G$.
In each case, the construction of ${\cal S}'$ makes sure that ${\cal X'}$ is a connected packing of $G$
and that the countries are updated in a way that their capitals remain inside them.
Moreover, the highways are updated so to remain internally disjoint and inside the corresponding updated countries. 
We next prove that $\Scal'$ is $\Acal$-normal.
Condition (A) is invariant as the cloud we take into consideration cannot intersect any $W \in {\cal W}$ and a cloud intersecting some capital $W \in {\cal W}$ cannot be disconnected from $W$.
It now remains to prove condition (B).
Because of Condition 4 of the definition of a $\Lambda$-state configuration, if a cloud $C$ intersects a freeway, then it shadows at least one state.
Now assume that a cloud $C$ intersects two freeways in $\Scal'$, then by construction of $\Scal'$, it also intersects at least the two same freeways in $\Scal$. This along with the fact that  $\Scal$ satisfies Condition (B), implies that $\Scal'$ satisfies condition (B) as well, as required.

Notice that, for any cloud $C^*\in{\cal A}\setminus \{C\}$, 
if $C^*$ does not intersect a state $X$ in ${\cal S}$, then the corresponding 
state $X'$ in ${\cal S}'$, i.e., the state $X'=\alpha'(\alpha^{-1}(X))$, also does not intersect $C^*$.
This means that ${\sf cost}({\cal S'},{\cal A}) \leq  {\sf cost}({\cal S},{\cal A})$.

Notice now that by the construction of ${\cal S}'$, $C$ is not in ${\sf front}_{\cal A}({\cal S}')$.
In the case where  ${\sf cov}_{\cal S}(C)\geq 1$ we have that
${\sf cost}({\cal S'},{\cal A}) < {\sf cost}({\cal S},{\cal A})$.

Notice that the case where 
${\sf cov}_{\cal S}(C)=0$ happens only when ${\sf action}={\sf anex}$ and 
there is an edge with one endpoint in $C$ and one in some country $X^*$ of ${\cal S}$ that 
does not intersect $C$. 
Moreover  ${\sf cc}_{G}(X\setminus C,\alpha^{-1}(X))=X$, for every state $X$ of ${\cal S}$.
This implies that ${\sf indep}(\Scal')\subseteq {\sf indep}(\Scal)$.
As $C\subseteq {\sf indep}(\Scal)$ and $C\cap {\sf indep}(\Scal')=\emptyset$, we conclude that 
$|{\sf indep}(\Scal')|<|{\sf indep}(\Scal)|$ as required.
\end{myproof}

To continue with the proof of \autoref{lemma_4} we explain how to transform the  ${\cal A}$-normal  $\Lambda$-state configuration 
${\cal S}$ of $G$ to a complete one. This is done in two phases. First, as long as there is an ${\cal A}$-cloud  $C\in{\sf front}_{\cal}({\cal S})$
where ${\sf cov}_{\cal S}({C})\geq 1$, we apply one of the above three procedures depending on the 
number of freeways intersected by $C$.  We again use ${\cal S}$ to denote the   ${\cal A}$-normal  $\Lambda$-state configuration of $G$ that is created in the end of this first phase. 
Notice that, as there is no ${\cal A}$-cloud with ${\sf cov}_{\cal S}({C})\geq 1$, then ${\sf cost}_{\cal A}({\cal S})=0$. 
The second phase  is the application of ${\sf anex}({\cal S},C)$,
as long as some $C\in {\sf front}_{\cal A}({\cal S})$ is touching some of the countries of ${\cal S}$.
We claim that this procedure will be applied as long as  there are vertices in ${\sf indep}(\Scal)$.
Indeed, if this is the case, the set ${\sf front}_{\cal A}({\cal S})$ is non-empty and 
by the connectivity of $G$, there is always a $C\in{\sf front}_{\cal A}({\cal S})$
that is touching some country of ${\cal S}$. Therefore, as ${\sf cost}_{\cal A}({\cal S})=0$ (by {Claim}~\ref{d0oi3d}), procedure
${\sf anex}({\cal S},C)$ will be applied again. 

By {Claim}~\ref{d0oi3d}, $|{\sf indep}(\Scal)|$ is strictly decreasing during the second phase. We again use ${\cal S}$ for the final outcome of this second phase. We  have that ${\sf indep}(\Scal)=\emptyset$
and we conclude that ${\cal S}$ is a complete ${\cal A}$-normal  $\Lambda$-state configuration of $G$ such that $|{\sf front}_{\cal A}({\cal S})| = 0$.

We  are now going to create a graph isomorphic to $\Lambda$ only by doing contractions in $G$.
For this we use $\Scal$, a complete ${\cal A}$-normal  $\Lambda$-state configuration of $G$ such that $|{\sf front}_{\cal A}({\cal S})| = 0$, obtained  as describe before.
We contract in $G$ every country of $\Scal$ into a unique vertex.
This can be done because the countries of $\Scal$ are connected.
Let $G'$ be the resulting graph.
By construction of $\Scal$, $G'$ is a contraction of $H$.
Because of Condition $4$ 
of $\Lambda$-state configuration, every freeway of $\Scal$ becomes an edge in $G'$.
This implies that there is a graph isomorphic to $\Lambda$ that is a subgraph of $G'$.
So  $\hat{\Gamma}_{k'}$ is isomorphic to a subgraph of  $G'$ with the same number of vertices.
Let see $\hat{\Gamma}_{k'}$ as a subgraph of $G'$ and let $e$ be an edge of $G'$ that is not an edge of  $\hat{\Gamma}_{k'}$.
As $e$ is an edge of $G'$, this implies that in $G$,  there is two states of  $\Scal$ such that there is no freeway between them but still an edge.
This is not possible by construction of $\Scal$.
We deduce that $G'$ is isomorphic to $\hat{\Gamma}_{k'}$.
Moreover, as $|{\sf front}_{\cal A}({\cal S})| = 0$, then every cloud is a subset of a country.
This implies that $G'$ is also a contraction of $H$.
By contracting in $G'$ the edge corresponding to 
{$\{a,({k'-1,k'-1})\}$} 
in $\hat{\Gamma}_{k'}$, we obtain that $\Gamma_{k'}$ is a contraction of $H$.
\autoref{lemma_4} follows.
\end{proof}

\begin{proof}[Proof of \autoref{j9io3nj78}]
  Let $\lambda$, $c$, $c_1$, and $c_2$ be integers. 
  It is enough to prove that there exists an integer $\lambda' = \Ocal(\lambda\cdot c_1\cdot ( c_2)^c)$ such that for every graph class $\Gcal \in \sqgc(c)$, 
\begin{eqnarray*}
\forall G \in \Gcal\  & \tw(G)  \leq  \lambda\cdot  ({\bf bcg}(G))^c  & \Rightarrow\\ 
\forall F \in \Gcal^{(c_1,c_2)}\  &  \tw(F)\leq \lambda'\cdot  ({\bf bcg}(F))^c. & 
\end{eqnarray*}

Let $\Gcal \in \sqgc(c)$ be a class of graph such that $\forall G \in \Gcal\ \ \tw(G)\leq \lambda\cdot  ({\bf bcg}(G))^c$.
Let $H \in \Gcal^{(c_1,c_2)}$ and let $G$ and $J$ be two graphs such that $G \in \Gcal$, $G \leq^{(c_1)} J \mbox{, and } H \leq^{c_2} J$. $G$ and $J$ exist by definition of $\Gcal^{(c_1,c_2)}$.

\begin{itemize}
\item By definition of $H$ and $J$, $\tw(H) \leq \tw(J)$.
\item By \autoref{lemma:twcont},  $\tw(J) \leq (c_1+1)(\tw(G)+1) -1$.
\item By definition of $\Gcal$,  $\tw(G) \leq \lambda\cdot \bcg(G)^c$.
\item By \autoref{lemma_4},  $\bcg(G)\leq (2 c_2+1) (\bcg(H)+2)+1$.
\end{itemize}
If we combine these four statements, we obtain that 
$$\tw(H) \leq (c_1+1)(\lambda\cdot [(2 c_2+1) (\bcg(H)+2)+1]^c +1) -1.$$
\noindent As the formula is independent of the graph class, the \autoref{j9io3nj78} follows.\end{proof}

\section{Conclusions, extensions, and open problems}

\label{copsed}

The main combinatorial result of this paper is that, for every $d$ and every apex-minor-free  graph class ${\cal G}$, the intersection class
${\sf inter}_{d}({\cal G})$ 
 has the SQG{\bf C} property for $c=1$.
%
Certainly, the main general question is to detect even wider graph classes with the SQG{\bf M}/SQG{\bf C} property.
In this direction, some insisting open issues are the following:

\begin{itemize}
\item Is the bound on the (multi-)degree necessary? Are there classes of intersection graphs
with unbounded or ``almost bounded''  maximum degree that have the SQG{\bf M}/SQG{\bf C} property?
\item All so far known results classify graph classes in \sqgm(1) or \sqgc(1). Are there (interesting) graph classes in $\sqgm(c)$ or $\sqgc(c)$ for some $1<c<2$ that do not belong in $\sqgm(1)$ or $\sqgc(1)$ respectively? An easy (but trivial) example of such a class 
is the class ${\cal Q}_{d}$ of the q-dimensional grids, i.e., the cartesian products of $q\geq 2$ equal length paths. 
It is easy to see that the maximum $k$ for which an $n$-vertex graph $G\in{\cal Q}_{q}$ contains a $(k\times k)$-grid as a minor is $k=\Theta(n^{\frac{1}{2}})$. On the other size, it can also be proven that    $\tw(G)=\Theta(n^{\frac{q-1}{q}})$. These two facts together 
imply that ${\cal Q}_{q}\in\sqgm(2-\frac{2}{q})$ while ${\cal Q}_{q}\not\in\sqgm(2-\frac{2}{q}-\epsilon)$ for every $\epsilon>0$.
\item Usually the graph classes in $\sqgc(1)$ are characterised by some ``flatness'' property. For instance, see the results in~\cite{GiTh2013,KaThWo2017,KaThWo2017} for $H$-minor free graphs, where $H$ is an apex graph. Can 
 $\sqgc(1)$ be useful as an intuitive definition of the ``flatness'' concept? Does this have some geometric interpretation?
\end{itemize}

{
}

\end{document}